\documentclass[11.1pt,reqno]{amsart}

\renewcommand{\emptyset}{\varnothing}
\newcommand{\shp}{shape}
\newcommand{\shd}{clique}

\usepackage{color}
\usepackage{amssymb}
\usepackage{amsmath}
\usepackage{amsthm}
\usepackage[pagebackref,colorlinks=true,urlcolor=blue,linkcolor=blue,citecolor=blue]{hyperref}
\newcommand{\N}{\mathbb{N}}
\newcommand{\Z}{\mathbb{Z}}
\newcommand{\Q}{\mathbb{Q}}

\newcommand{\PP}{\mathbb{P}}
\newcommand{\F}{\vec F}
\newcommand{\vH}{\vec H}
\newcommand{\E}{\operatorname{End}}

\newcommand{\HH}{\operatorname{Hom}}

\newtheorem{theorem}{Theorem}[section]
\newtheorem{definition}[theorem]{Definition}
\newtheorem{lemma}[theorem]{Lemma}
\newtheorem{corollary}[theorem]{Corollary}
\newtheorem{proposition}[theorem]{Proposition}

\newtheorem{problem}[theorem]{Problem}
\newtheorem{remark}[theorem]{Remark}
\newtheorem{example}[theorem]{Example}
\renewcommand{\emph}[1]{\textbf{#1}}

\begin{document}

\title[Polynomial extensions of partition results]{New polynomial and multidimensional extensions of classical partition results}

\author[Bergelson]{Vitaly Bergelson}
\author[Johnson]{John H.~Johnson Jr.}
\author[Moreira]{Joel Moreira}


\thanks{The first author gratefully acknowledges the support of the NSF under grant DMS-1162073}
\begin{abstract}
In the 1970s Deuber introduced the notion of $(m,p,c)$-sets in $\N$ and showed that these sets are partition regular and contain all linear partition regular configurations in $\N$.
  In this paper we obtain enhancements and extensions of classical results on $(m,p,c)$-sets in two directions.
  First, we show, with the help of ultrafilter techniques, that Deuber's results extend to polynomial configurations in abelian groups.
  In particular, we obtain new partition regular polynomial configurations in $\Z^d$.
  Second, we give two proofs of a generalization of Deuber's results to general commutative semigroups.

  We also obtain a polynomial version of the central sets theorem of Furstenberg, extend the theory of $(m,p,c)$-systems of Deuber, Hindman and Lefmann and generalize a classical theorem of Rado regarding partition regularity of linear systems of equations over $\N$ to commutative semigroups.
\end{abstract}
\maketitle

\section{Introduction}

The main goal of this paper is to obtain new polynomial and multidimensional generalizations of Ramsey-theoretical results due to R. Rado \cite{Rado33} and W. Deuber \cite{Deuber73}.
To put our results into perspective, we will start the discussion by briefly reviewing some of the relevant classical results.

Some familiar results of Ramsey theory can be formulated as results about partition regularity of homogeneous systems of equations. For example, the celebrated van der Waerden theorem \cite{vdWaerden27}, which states that, for any finite coloring $\N=\bigcup_{i=1}^rC_i$, one of the $C_i$ contains arbitrarily long arithmetic progressions $\{x, x+d,\dots,x+(k-1)d\},~d\neq0$, can be formulated as follows.
\begin{theorem}
  For any finite coloring of $\N=\{1,2,\dots\}$ and for any $k\in\N$ there exists a monochromatic solution of the system
  \begin{equation}\label{eq_intro_vdw}x_2-x_1=x_3-x_2=\cdots=x_k-x_{k-1}\neq0.\end{equation}
\end{theorem}

A slightly stronger theorem, due to A. Brauer \cite{Brauer28}, states that one can actually guarantee that the difference $d$ of the monochromatic progression $\{x, x+d,\dots,x+(k-1)d\}$ appearing in van der Waerden's theorem is also of the same color. Since $x,d,x+d$ satisfy the equation $x+y=z$, it follows that Brauer's theorem is a simultaneous extension of Schur's theorem \cite{Schur16} (which states that $x+y=z$ is a partition regular equation over $\N$) and van der Waerden's theorem. Here is a formulation of Brauer's theorem in the language of partition regularity of systems of homogeneous equations.
\begin{theorem}\label{thm_introbrauer}
  For any $k\in\N$, the system
  \begin{equation}\label{eq_brauer}\left\{\begin{array}{ccc}x_2-x_1&=&x_0\\ \vdots&\vdots&\vdots\\ x_k-x_{k-1}&=&x_0\end{array}\right.
  \end{equation}
  is partition regular, meaning that, for any partition $\N=\bigcup_{i=1}^rC_i$, one of the $C_i$ contains a solution $(x_0,x_1,\dots,x_k)$ of (\ref{eq_brauer}).
\end{theorem}
In his fundamental paper \cite{Rado33}, R. Rado established a necessary and sufficient condition for partition regularity of the system $C{\bf x}=0$, where $C$ is a $k\times n$ matrix with integer entries and ${\bf x}$ is an $n$-dimensional vector. For the formulation of Rado's theorem see Section \ref{section_rado} below.

In 1973, W. Deuber offered a new approach to partition regularity of homogeneous systems of linear equations \cite{Deuber73}.
The main novelty of Deuber's approach was the introduction of a family of configurations, the so-called $(m,p,c)$-sets defined in Definition \ref{def_intrompc} below. On the one hand, these configurations can always be found in one cell of a partition of $\N$, while on the other they contain solutions of homogeneous partition regular systems of equations.

\begin{definition}\label{def_intrompc}
Let $m,p,c\in\N$ and let ${\bf s}=(s_0,\dots,s_m)\in(\Z\setminus\{0\})^{m+1}$. The $(m,p,c)$-set generated by ${\bf s}$ is the set
$$D(m,p,c;{\bf s})=\left\{\begin{array}{lr}cs_0,&\\is_0 +cs_1,& i\in\{-p,\dots,p\}\\is_0 +js_1+cs_2,& i,j\in\{-p,\dots,p\}\\ \qquad\vdots&\vdots\qquad\qquad\\i_0s_0 +\cdots+i_{m-1}s_{m-1}+cs_m,& i_{m-1},\dots,i_0\in\{-p,\dots,p\}
\end{array}\right\}$$
\end{definition}
The following theorem summarizes Deuber's results from \cite{Deuber73}.
\begin{theorem}\label{thm_introdeuber}\

  \begin{enumerate}
    \item For any $m,p,c\in\N$ and any finite partition $\N=\bigcup_{i=1}^rC_i$, one of the $C_i$ contains an $(m,p,c)$-set for some ${\bf s}\in\N^{m+1}$.
    \item For any $m,p,c,r\in\N$, there exist $M,P,C\in\N$ such that for any ${\bf S}\in\N^{M+1}$ and any $r$-coloring of $D(M,P,C;{\bf S})$, there exists ${\bf s}\in\N^{m+1}$ such that $D(m,p,c;{\bf s})$ is monochromatic.
  \end{enumerate}
\end{theorem}
Theorem \ref{thm_introdeuber} contains as special cases several classical Ramsey-theoretical results:
\begin{example}\label{example_intro}
\leavevmode
  \begin{enumerate}
    \item Schur's theorem (stated above). Indeed, any $(1,1,1)$-set contains elements $s_0,s_1,s_0+s_1$.
    \item Brauer's theorem (Theorem \ref{thm_introbrauer} above). Indeed, any $(1,k,1)$-set contains elements $s_0,s_1,s_1+s_0,\dots,s_1+ks_0$ which satisfy (\ref{eq_brauer}). As a consequence, van der Waerden's theorem also follows from Deuber's result.
    \item\label{item_example_intro} Folkman's theorem (cf. \cite[Theorem 3.11]{Graham_Rothschild_Spencer80}), stating that for any finite coloring of $\N$ there exists a set $A$ of arbitrary finite cardinality such that the set $FS(A):=\{\sum_{i\in B}i:\emptyset\neq B\subset A\}$ is monochromatic. $FS(A)$ is contained in a $(m,1,1)$-set, where $m+1$ is the cardinality of $A$.
  \end{enumerate}
\end{example}
The theorems mentioned in Example \ref{example_intro} also follow from Rado's criterion \cite{Rado33} for partition regularity of a system of linear equations (see Theorem \ref{thm_rado} below) although not as immediately.
In fact Deuber proved that a set $A\subset\N$ contains an $(m,p,c)$-set for every $m,p,c\in\N$ if and only if $A$ contains a solution to every partition regular system of the form $C{\bf x}=0$.

Deuber's approach allowed him to confirm a conjecture of Rado, stated in \cite{Rado33}. To formulate Rado's conjecture, call a set $A\subset\N$ \emph{rich} if it contains a solution to every partition regular homogeneous system of linear equations.
One can reformulate Rado's theorem as ``for any finite partition of $\N$, one of the cells is rich".
Rado's conjecture stated that for any finite partition of a rich set, one of the cells is still rich; this conjecture follows from part (2) of Theorem \ref{thm_introdeuber}.

We will see below that Theorem \ref{thm_introdeuber} can be significantly generalized in two ways.
On the one hand, we will see that results similar to Theorem \ref{thm_introdeuber} can be proved for general countable commutative semigroups (see Theorem \ref{thm_introgeneraldeuber} below).
On the other hand, in the case of countable abelian groups, part (1) of Theorem \ref{thm_introdeuber} admits a polynomial generalization (see Theorem \ref{cor_introduction} below).
These generalizations hinge on a broadening of the notion 
of $(m,p,c)$-sets (see Definition \ref{def_deuberset} and the discussion that follows it in Section \ref{section_statements}).
The following definition, which is a special case of Definition \ref{def_deuberset},
gives the flavor of the idea behind generalized $(m,p,c)$-sets.
\begin{definition}\label{def_introdeuber}
  Let $m,d\in\N$, let $c:\Z^d\to\Z^d$ be an additive homomorphism, and let $\F=(F_1,\dots,F_m)$ be an $m$-tuple\footnote{Throughout this paper, we use the arrow notation $\F$ for tuples of sets of functions. For tuples of semigroup elements $(s_1,\dots,s_d)\in G^d$ we will use boldface: ${\bf s}=(s_1,\dots,s_d)$.} where for each $i=1,\dots,m$, $F_i$ is a finite family of polynomial functions of the form $f:\Z^{id}\to\Z^d$ such that $f(0)=0$.
  Finally, let ${\bf s}=(s_0,\dots,s_m)\in(\Z^d\setminus\{0\})^{m+1}$.
  Then the $(m,\F,c)$-set generated by ${\bf s}$ is defined by
  $$D(m,\F,c;{\bf s}):=\left\{\begin{array}{lr}c(s_0)&\\f(s_0)+c(s_1),& f\in F_1\\f(s_0,s_1)+c(s_2),& f\in F_2\\ \vdots&\vdots\\f(s_0,\dots,s_{m-1})+c(s_m),& f\in F_m
\end{array}\right\}$$
\end{definition}
Note that, when $d=1$ and when all the polynomials are linear, Definition \ref{def_introdeuber} reduces to Definition \ref{def_intrompc}. Indeed, given a triple $(m,p,c)\in\N^3$ one can let $\tilde c$ be the map defined by $\tilde c:x\mapsto cx$ for each $x\in\Z$ and, for each $j=1,\dots,m$, let $F_j$ be the set of all maps ${\bf x}\mapsto\langle {\bf x},\xi\rangle$ with $\xi\in\{-p,\dots,p\}^j$ and ${\bf x}\in\Z^j$. Finally make $\F=(F_1,\dots,F_m)$.
Then an $(m,p,c)$-set is a $(m,\F,\tilde c)$-set.

The following is a multidimensional generalization of Theorem \ref{thm_introdeuber}.
 The first (resp. second) part of Theorem \ref{thm_introgeneraldeuber} is a special case of the more technical Corollary \ref{thm_statements_main} (resp. Theorem \ref{thm_finitistic}), which is proved in Section \ref{section_mpcincentral} (resp. \ref{section_finitistic}), and extends the first (resp. second) part of Theorem \ref{thm_introdeuber}.

\begin{theorem}\label{thm_introgeneraldeuber}Let $d,m\in\N$, let $c:\Z^d\to\Z^d$ be a scalar homomorphism (i.e. $c(x_1,\dots,x_d)=(ax_1,\dots,ax_d)$ for some $a\in\Z\setminus\{0\}$) and let $\F=(F_1,\dots,F_m)$ be an $m$-tuple where, for each $i=1,\dots,m$, $F_i$ is a finite family of homomorphisms from $(\Z^d)^i$ to $\Z^d$.

  \begin{enumerate}
    \item For any finite partition $\Z^d=\bigcup_{i=1}^rC_i$, one of the $C_i$ contains $D(m,\F,c;{\bf s})$ for some ${\bf s}=(s_0,\dots,s_m)\in(\Z^d\setminus\{0\})^{m+1}$.
    \item For any $r\in\N$, there exist $M\in\N$, a scalar homomorphism $C:\Z^d\to\Z^d$ and an $M$-tuple $\vH=(H_1,\dots,H_m)$ where $H_i$ is a finite family of homomorphisms from $\Z^{di}$ to $\Z^d$ such that for any ${\bf S}\in(\Z^d\setminus\{0\})^{M+1}$ and any $r$-coloring of $D(M,\vH,C,{\bf S})$, there exists ${\bf s}\in(\Z^d\setminus\{0\})^{m+1}$ such that $D(m,\F,c;{\bf s})$ is a subset of $D(M,\vH,C,{\bf S})$ and is monochromatic.
  \end{enumerate}
\end{theorem}
We remark that when all the homomorphisms are scalar, Theorem \ref{thm_introgeneraldeuber} can be derived from \cite{Deuber75}.

The following result, which is a special case of part (2) of Corollary \ref{thm_statements_main} below, can be viewed as a polynomial extension of part (1) of Theorem \ref{thm_introdeuber}.
\begin{theorem}\label{cor_introduction}
  Let $d,m\in\N$ and, for each $i=1,2,\dots,m$ let $F_i$ be a finite set of polynomials of the form $f:(\Z^d)^i\to\Z^d$ such that $f({\bf 0})={\bf0}$.
  Let $\F=(F_1,\dots,F_m)$ and let $c:\Z^d\to\Z^d$ be a scalar homomorphism.
  For any finite coloring of $\Z^d$, there exists ${\bf s}\in(\Z\setminus\{0\})^{m+1}$ such that the set $D(m,\F,c;{\bf s})$ is monochromatic.
\end{theorem}

Van der Waerden's theorem was generalized to higher dimensions by T. Gr\"unwald (Galai)\footnote{R. Rado attributed this result to G. Gr\"unwald in \cite[p. 123]{Rado43}. Rado, however, had in mind T. Gr\"unwald who never published his proof and later changed his name to Galai.
(G\'eza Gr\"unwald was a talented young Hungarian analyst who was murdered in 1943, see http://www.math.technion.ac.il/hat/people/obits).}.
The following corollary is a simultaneous generalization of Brauer's theorem and of the multidimensional extension of van der Waerden's theorem. (It will be proved in Section \ref{section_statements} after Corollary \ref{thm_statements_main}).

\begin{corollary}\label{cor_intro_multidimbrauer}
Let $d\in\N$ and let $f:\N^d\to\N$ be a semigroup homomorphism\footnote{Here and in the rest of this paper, a semigroup homomorphism is a map $f:G\to H$, where $G$ and $H$ are commutative semigroups, such that $f(a+b)=f(a)+f(b)$ for all $a,b\in G$.} (here $\N$ is a shorthand for $(\N,+)$).
For any finite partition of $\N^d$ and any $k\in\N$ there exist $a,b\in\N^d$ such that the set
\begin{equation}
  \{b\}\cup\Big\{a+\big(i_1f(b),\cdots,i_df(b)\big):0\leq i_1,\dots,i_d\leq k\Big\}
\end{equation}
is contained in a single cell of the partition.
\end{corollary}

Observe that when $d=1$ this result reduces to Brauer's theorem.

A polynomial generalization of the multidimensional van der Waerden theorem was established in \cite{Bergelson_Leibman96}.
An immediate corollary of Theorem \ref{cor_introduction} (corresponding to $m=1$) is the following common generalization of Brauer's theorem and the multidimensional polynomial van der Waerden theorem.

\begin{corollary}\label{cor_polybrauer}
Let $d\in\N$ and let $F$ be a finite set of polynomials $f:\Z^d\to\Z$ such that $f(0)=0$.
For any finite partition of $\Z^d$ there exist $a,b\in(\Z\setminus\{0\})^d$ such that the set
$$\{b\}\cup\Big\{a+\big(f_1(b),\cdots,f_d(b)\big):f_1,\dots,f_d\in F\Big\}$$
is contained in a single cell of the partition.
\end{corollary}

This corollary can also be deduced from \cite[Theorem 0.11]{Bergelson_McCutcheon00}.
For $m=2$, Theorem \ref{cor_introduction} (and the more general Corollary \ref{thm_statements_main}) provide new classes of partition regular configurations.
For example, it follows from Theorem \ref{cor_introduction} that for any finite coloring of $\N$ there exists a monochromatic quadruple $\{x,y+x^2,z,z+y^2\}$.
To see this, apply Theorem \ref{cor_introduction} with $d=1$, $m=2$, $F_1$ containing only the polynomial $x\mapsto x^2$ and $F_2$ comprised of the polynomials $(x,y)\mapsto0$ and $(x,y)\mapsto y^2$.

A more general corollary of Theorem \ref{cor_introduction} is the following result, which involves a ``chain of configurations'' of the form $\{x,y,x+f(y)\}$ where $f$ is a polynomial.

\begin{corollary}\label{cor_IntroChainSarkozy}
Let $k\in\N$ and let $f_1,\dots,f_k\in\Z[x]$ be polynomials with $f_i(0)=0$.
Then for any finite coloring of $\Z$ there exist $x_0,x_1,\dots,x_k,a_1,a_2,\dots,a_k\in\Z\setminus\{0\}$, all with the same color, satisfying
$$\left\{\begin{array}{ccc}a_1-x_1&=&f_1(x_0)\\ a_2-x_2&=&f_2(x_1)\\ \vdots&&\vdots\\ a_k-x_k&=&f_k(x_{k-1})\end{array}\right.$$
\end{corollary}
Corollary \ref{cor_IntroChainSarkozy} follows from Theorem \ref{cor_introduction} by putting $d=1$, $m=k$ and, for each $i=1,\dots,k$, letting $F_i$ consist of the zero polynomial and the polynomial $(x_0,\dots,x_{i-1})\mapsto f_i(x_{i-1})$.

Another new result obtained in this paper is a polynomial extension of Furstenberg's central sets theorem \cite[Proposition 8.21]{Furstenberg81} which is of independent interest (see Theorem \ref{thm_GCST} below) and is essential to the proofs of some combinatorial results below.
We postpone its formulation to a later section as it requires some additional definitions to state.
Other important tools employed in our proofs include the polynomial Hales-Jewett theorem \cite{Bergelson_Leibman96} and the IP-polynomial Szemer\'edi theorem \cite{Bergelson_McCutcheon00}.

The paper is organized as follows.
In Section \ref{sec_preliminares} we review the necessary background material.
In Section \ref{section_statements} we give precise definitions and formulations of our results.
In Sections \ref{section_mpcincentral} and \ref{section_finitistic} we prove our generalizations of Theorem \ref{thm_introdeuber}, namely Corollary \ref{thm_statements_main} and Theorem \ref{thm_finitistic}.
In Section \ref{sec_mfcsystems} we extend results of Deuber, Hindman and Lefmann on $(m,p,c)$-systems, which are common extensions of Deuber's result and Hindman's theorem\cite{Deuber_Hindman87,Hindman_Lefmann93}.
Finally, in Section \ref{section_rado} we derive, in the spirit of Rado's theorem, a rather general sufficient condition for partition regularity of a system of linear equations in a countable commutative semigroup.
\subsection*{Acknowledgement}
The authors wish to thank Donald Robertson for multiple useful remarks on an early draft of the paper.

\section{Preliminaries}\label{sec_preliminares}

\subsection{IP-sets}
Given an infinite set $X$, we denote by ${\mathcal F}(X)$ the family of all finite non-empty subsets of $X$, i.e., ${\mathcal F}(X):=\{\alpha\subset X:0<|\alpha|<\infty\}$.
We denote by ${\mathcal  F}={\mathcal F}(\N)$ the family of all non-empty finite subsets of $\N$.
Let $G$ be a countable commutative semigroup and let $(x_n)_{n\in\N}$ be an injective sequence in $G$.
For each $\alpha\in{\mathcal  F}$ define $x_\alpha=\sum_{n\in\alpha}x_n$.
The \emph{IP-set} generated by $(x_n)_{n\in\N}$ is the set $FS(x_n)=\{x_\alpha:\alpha\in{\mathcal  F}\}$.
Clearly $x_{\alpha\cup\beta}=x_\alpha+x_\beta$ for any disjoint $\alpha,\beta\in{\mathcal  F}$. Moreover, if $(y_\alpha)_{\alpha\in{\mathcal  F}}$ is any `sequence' indexed by ${\mathcal  F}$ such that $x_{\alpha\cup\beta}=x_\alpha+x_\beta$ for any disjoint $\alpha,\beta\in{\mathcal  F}$, then the set $\{y_\alpha:\alpha\in{\mathcal  F}\}$ is an IP-set (generated by $(y_{\{n\}})_{n\in\N})$).
For this reason we will denote IP-sets by $(y_\alpha)_{\alpha\in{\mathcal  F}}$, with the understanding that they are generated by the singletons $y_n$, $n\in\N$.

\begin{definition}
Let $(x_\alpha)_{\alpha\in{\mathcal  F}},(y_\alpha)_{\alpha\in{\mathcal  F}}$ be IP-sets in a countable commutative semigroup $G$.
\begin{enumerate}
  \item For $\alpha,\beta\in{\mathcal  F}$ we write $\alpha<\beta$ as a shortcut to $\max_{i\in\alpha}i<\min_{j\in\beta}j$.
  \item We say that $(x_\alpha)_{\alpha\in{\mathcal  F}}$ is a sub-IP-set of $(y_\alpha)_{\alpha\in{\mathcal  F}}$ if there exist $\alpha_1<\alpha_2<\cdots$ in ${\mathcal F}$ such that $x_n=y_{\alpha_n}$ for all $n\in\N$.
\end{enumerate}
\end{definition}

\subsection{Central sets and $D$-sets}
Central sets were introduced by Furstenberg in $(\N,+)$ in \cite{Furstenberg81}.
A characterization in terms of ultrafilters was discovered later \cite{Bergelson_Hindman90}, and this spurred the study of central sets.
For the reader's convenience we will state some of the basic properties of ultrafilters that we will use.
The reader will find missing details in \cite{Bergelson10} or \cite{Hindman_Strauss98}.
\begin{definition}
A \emph{filter} on a countable set $G$ is a non-empty family $p$ of subsets of $G$ such that
\begin{enumerate}
\item $\emptyset\notin p$.
\item If $A\in p$ and $A\subset B$ then $B\in p$.
\item If $A$ and $B$ are both in $p$ then $A\cap B\in p$.
\end{enumerate}
If in addition $p$ satisfies the following condition, then $p$ is an \emph{ultrafilter}.
\begin{enumerate}\setcounter{enumi}{3}
\item $A\in p\iff (G\setminus A)\notin p$.
\end{enumerate}
\end{definition}

\begin{remark}\label{rmrk_ultrafilter}
Equivalently, an ultrafilter is a family $p$ of subsets of $G$ such that for any finite partition of $G$, exactly one of the cells of the partition belongs to $p$.
\end{remark}
The simplest example of an ultrafilter is that of a \emph{principal ultrafilter} $p_g$ generated by a point $g\in G$ and defined by $A\in p_g\iff g\in A$.
In fact, these are the only explicit examples; the existence of non principal ultrafilters needs some form of the axiom of choice.

Ultrafilters are maximal filters (with respect to the inclusion relation) and hence, by Zorn's lemma, any filter is contained in an ultrafilter.
The set of all ultrafilters on $G$ is denoted by $\beta G$ and can be identified with the Stone-\v{C}ech compactification of the (discrete) space $G$ (see, for example, Theorem 3.27 in \cite{Hindman_Strauss98}). The space $\beta G$ is a compact Hausdorff space with the topology generated by the clopen sets
\begin{equation}\label{eq_closure}\overline{A}:=\{p\in\beta G:A\in p\}\qquad\qquad\forall A\subset G\end{equation}

One can naturally extend the semigroup operation from $G$ to $\beta G$.
When $A\subset G$ and $g\in G$ we use the notation $A-g:=\{h\in G:h+g\in A\}$.
Given $p,q\in\beta G$ we define
\begin{equation}\label{eq_ultrasum}p+q=\{A\subset G:\{g\in G:A-g\in p\}\in q\}
\end{equation}

The operation defined in (\ref{eq_ultrasum}) is associative (cf. Theorems 4.1, 4.4 and 4.12 in \cite{Hindman_Strauss98}) but, in general, not commutative.
An ultrafilter $p\in\beta G$ is called \emph{idempotent} if $p+p=p$.
By a theorem of Ellis \cite{Ellis58}, any semi-continuous compact semigroup contains an idempotent, so in particular for any countable semigroup $G$ there exists an idempotent ultrafilter in $\beta G$.
The interest in idempotent ultrafilters lies in the fact that any set belonging to such an ultrafilter contains an IP-set; this fact implies Hindman's celebrated theorem \cite{Hindman74} stating that for any finite partition of $\N$, one of the cells contains an IP-set (cf. \cite[Sections 2 and 3]{Bergelson10}).

A \emph{right ideal} in $\beta G$ is a subset $I\subset\beta G$ satisfying $I+\beta G\subset I$.
By Zorn's Lemma, there exist minimal (with respect to the inclusion relation) right ideals in $\beta G$.
A \emph{minimal} ultrafilter is an ultrafilter $p\in\beta G$ which belongs to some minimal right ideal.
To better understand the importance of minimal ultrafilters, we need the notion of piecewise syndetic sets.
\begin{definition}
  Let $G$ be a countable commutative semigroup and let $A\subset G$.
  \begin{enumerate}
    \item $A$ is a \emph{syndetic set} if finitely many shifts of $A$ cover $G$.
    More precisely, if there exists a finite set $F\subset G$ such that $G=\bigcup_{g\in F}(A-g)$.
    \item $A$ is a \emph{thick set} if it contains a shift of every finite set, i.e., if for every finite set $F\subset G$ there exists $g\in G$ such that $g+F\subset A$.
    \item $A$ is a \emph{piecewise syndetic set} if it is the intersection of a thick set with a syndetic set.
    In other words, $A$ is a piecewise syndetic set if there exists a finite set $F\subset G$ such that the union $\bigcup_{g\in F}(A-g)$ is thick.
  \end{enumerate}

\end{definition}
One can show that if $p\in\beta G$ is a minimal ultrafilter and $A\in p$, then $A$ is piecewise syndetic.
Conversely, for any piecewise syndetic set $A$, there exist minimal ultrafilters $p\in\beta G$ for which $A\in p$ (see, for example, \cite{Bergelson03}).
Of special importance among minimal ultrafilters are the \emph{minimal idempotent ultrafilters} i.e.\;ultrafilters which are simultaneously minimal and idempotent.
For any countable commutative semigroup $G$ there are minimal idempotent ultrafilters $p\in\beta G$.

\begin{definition}
  Let $G$ be a countable commutative semigroup and let $A\subset G$. We say that $A$ is a \emph{central set} if there exists a minimal idempotent ultrafilter $p\in\beta G$ such that $A\in p$.
\end{definition}

Since every countable commutative semigroup has a minimal idempotent, it follows from Remark \ref{rmrk_ultrafilter} that for every finite partition of a countable commutative semigroup, one of the cells is a central set.
Central sets are important in combinatorics because they are both IP-sets and piecewise syndetic sets; the combinatorial richness possessed by central sets is best illustrated by the central sets theorem.
\begin{theorem}[Central sets theorem]\label{thm_cst}
  Let $G$ be a countable commutative semigroup, let $j\in\N$, let $A\subset G$ be a central set and let $(y_\alpha)_{\alpha\in{\mathcal F}}$ be an IP-set in $G^j$.
  Then there exists an IP-set $(x_\beta)_{\beta\in{\mathcal F}}$ in $G$ and a sub-IP-set $(z_\beta)_{\beta\in{\mathcal F}}$ of $(y_\alpha)_{\alpha\in{\mathcal F}}$ such that
  $$\forall i\in\{1,\dots,j\}\qquad\forall\beta\in{\mathcal F}\qquad\qquad x_\beta+\pi_i(z_\beta)\in A$$
  where $\pi_i:G^j\to G$ is the projection onto the $i$-th coordinate.
\end{theorem}
This theorem was obtained by Furstenberg for the case $G=\N$ in \cite{Furstenberg81}.
In \cite{Bergelson_Hindman90}, Theorem \ref{thm_cst} was proved for certain classes of countable commutative semigroups, and an alternative, dynamical characterization of central sets for arbitrary countable commutative semigroups was establish, which hinted at the full generality of Theorem \ref{thm_cst}.
Theorem \ref{thm_cst} was obtained in full generality in \cite{Hindman_Maleki_Strauss96}.

By relaxing the definition of central set one obtains the notion of a $D$-set, which was introduced in \cite{Bergelson_Downarowicz08}.
While this notion makes sense in any countable amenable semigroup, we will only consider $D$-sets in $\Z^n$.
An ultrafilter $p\in\beta\Z^n$ is an \emph{essential idempotent} if it is an idempotent ultrafilter and every $A\in p$ has positive Banach upper density, i.e.

$$d^*(A)=\sup_{\{\Pi_k\}_{k\in\N}}\limsup_{k\to\infty}\frac{|A\cap\Pi_k|}{|\Pi_k|}>0$$
where the supremum is taken over all sequences of parallelepipeds
$$\Pi_k=[a_k^{(1)},b_k^{(1)}]\times\cdots\times[a_k^{(n)},b_k^{(n)}]\subset\Z^n;\ k\in\N$$
with $b_k^{(i)}-a_k^{(i)}\to\infty$ as $k\to\infty$ for all $1\leq i\leq n$.
\begin{definition}
   A set $A\subset\Z^n$ is a \emph{$D$-set} if there exists an essential idempotent $p\in\beta\Z$ such that $A\in p$.
\end{definition}
Every piecewise syndetic set has positive Banach upper density, therefore every central set is a $D$-set.
It was shown in \cite {Bergelson_Downarowicz08} that the converse is not true.
 However, the central sets theorem is true
 under the weaker assumption that $A$ is a $D$-set \cite{Beiglbock_Bergelson_Downarowicz_Fish09}.
Observe that, for every finite partition of $\Z$, one of the cells is a $D$-set.

\subsection{Some results we use}\label{subsec_statements}
In the course of our proofs we will take advantage of some powerful theorems.
For the convenience of the reader we list them in this subsection, but before we need a definition.
\begin{definition}\label{def_polmaps}
  Given a map $f:H\to G$ between countable commutative groups we say that $f$ is a \emph{polynomial map of degree 0} if it is constant.
We say that $f$ is a \emph{polynomial map of degree d}, $d\in\N$, if it is not a polynomial map of degree $d-1$ and for every $h\in H$, the map $x\mapsto f(x+h)-f(x)$ is a polynomial of degree $\leq d-1$.
Finally we denote by $\PP(G,H)$ the set of all polynomial maps $f:G\to H$ with $f(0)=0$.
\end{definition}
Note that homomorphisms are elements of $\PP(G,H)$ having degree $1$.
\begin{theorem}[Multidimensional IP polynomial Szemer\'edi theorem, \cite{Bergelson_McCutcheon00}, Theorem 0.10]\label{thm_mipsz}
Let $n\in\N$, let $B\subset\Z^n$ have positive Banach upper density, let $j\in\N$ and let $(y_\alpha)_{\alpha\in{\mathcal F}}$ be an IP-set in $(\Z^n)^j=\Z^{nj}$.
For any finite family $F\subset\PP(\Z^{nj},\Z^n)$ there exist $x\in\Z^n$ and $\alpha\in{\mathcal F}$ such that $x+f(y_\alpha)\in B$ for all $f\in F$.
\end{theorem}

\begin{theorem}[IP polynomial van der Waerden theorem for abelian groups, cf. {\cite[Corolary 8.8]{Bergelson_Leibman99}}]\label{thm_ippolvdwtrue}
Let $G,H$ be countable abelian groups and let $F\subset\PP(H,G)$ be a finite subset.
Then for every finite partition $G=C_1\cup\cdots\cup C_r$ and every IP set $(y_\alpha)_{\alpha\in{\mathcal F}}$ in $H$ there exists $i\in\{1,\dots,r\}$, $a\in C_i$ and $\alpha\in{\mathcal F}$ such that $a+f(y_\alpha)\in C_i$ for every $f\in F$.
\end{theorem}

\begin{theorem}[Equivalent finitistic form of Theorem \ref{thm_ippolvdwtrue}]\label{cor_finitisticippolvdw}
Let $r\in\N$, let $G,H$ be countable abelian groups, let $F\subset\PP(H,G)$ be a finite subset and let $(y_\alpha)_{\alpha\in{\mathcal F}}$ be an IP set in $H$.
There exists a finite set $I\subset G$ such that for every $r$-coloring\footnote{Here and throughout the paper, for $r\in\N$, we denote by $[r]$ the set $\{1,\dots,r\}$.} $\chi:I\to[r]$ there exists $i\in\{1,\dots,r\}$, $a\in I$ with $\chi(a)=i$ and $\alpha\in{\mathcal F}$ such that $\chi\big(a+f(y_\alpha)\big)=i$ for every $f\in F$.
\end{theorem}

\begin{proof}

It is easy to see that Theorem \ref{cor_finitisticippolvdw} implies Theorem \ref{thm_ippolvdwtrue}.
To prove the other direction,
  assume, for the sake of a contradiction, that Theorem \ref{cor_finitisticippolvdw} is false.
  Therefore for each finite subset $I\subset G$ there exists a ``bad'' coloring $\chi:I\to[r]$, that is, a coloring for which no monochromatic configuration of the form $\{a\}\cup\{a+f(y_\alpha):f\in F\}$ exists.
  Let $g_1,g_2,\dots$ be an enumeration of $G$ and, for each $n\in\N$, let $I_n=\{g_1,\dots,g_n\}$.
  Assume $\chi_n:I_n\to[r]$ is a ``bad'' r-coloring for $I_n$.

  Next we define a coloring $\chi:G\to[r]$.
  Let $S_0=\N$ and choose, inductively, for each $j\in\N$, some $i\in[r]$ for which the set $S_j:=\{n\in S_{j-1}:\chi_n(g_j)=i\}$ is infinite. Define $\chi(g_j)=i$.
  The coloring $\chi$ induces a partition of $G$ into $r$ sets.
  In view of Theorem \ref{thm_ippolvdwtrue}, there exists $i\in\{1,\dots,r\}$, $a\in G$ with $\chi(a)=i$ and $\alpha\in{\mathcal F}$ such that $\chi\big(a+f(y_\alpha)\big)=i$ for every $f\in F$.
  Since $F$ is finite, there exists some $j\in\N$ for which $a\in I_j$ and $a+f(y_\alpha)\in I_j$ for every $f\in F$.
  For $n\in S_j$, the coloring $\chi_n$ and $\chi$ agree on the set $\{a\}\cup\{a+f(y_\alpha):f\in F\}$.
  This contradicts the hypothesis that the coloring $\chi_n$ was ``bad'', which finishes the proof.
\end{proof}
We are now in position to prove the following statement, which will be utilized in the proof of Theorem \ref{thm_statements_central}.
\begin{corollary}\label{thm_ippolvdW}
  Let $j\in\N$, let $G$ be a countable abelian group and let $F$ be a finite family of polynomial maps from $G^j$ to $G$ such that $f({\bf0})=0$ for each $f\in F$.
  Then for every piecewise syndetic (in particular, central) set $A\subset G$ and every IP set $(y_\alpha)_{\alpha\in{\mathcal F}}$ in $G^j$ there exists $a\in A$ and $\alpha\in{\mathcal F}$ such that $a+f(y_\alpha)\in A$ for every $f\in F$.
\end{corollary}
\begin{proof}
Since $A$ is piecewise syndetic there exists a finite set $J\subset G$ such that $T:=A-J$ is thick.
Take $r=|J|$ and apply Theorem \ref{cor_finitisticippolvdw}; let $I$ be the finite set obtained.
Since $T$ is thick, there exists some $g\in G$ such that $I+g\subset T$.
Let $\chi:I\to J$ be defined so that $x+g+\chi(x)\in A$ for all $x\in I$.
Since $|J|=r$, there exists some $j\in J$, $\tilde a\in I$ with $\chi(\tilde a)=j$ and $\alpha\in{\mathcal F}$ such that $\chi\big(\tilde a+f(y_\alpha)\big)=j$ for all $f\in F$.

Using the definition of $\chi$, we conclude that $a:=\tilde a+g+j\in A$ and, for every $f\in F$, we have $a+f(y_\alpha)=\tilde a+f(y_\alpha)+g+j\in A$.
\end{proof}

\begin{definition}[Combinatorial line]
Let $A$ be a finite alphabet, let ${*}\notin A$ and let $n\in\N$.
A \emph{variable word} in $A^n$ is an element of the set $(A\cup\{*\})^n\setminus A^n$.
Given a variable word $w$ and $a\in A$ let $w(a)\in A^n$ be the word obtained by replacing each instance of $*$ in $w$ with $a$.
The \emph{combinatorial line} generated by a variable word $w$ is the set $\{w(a):a\in A\}\subset A^n$.
\end{definition}
\begin{theorem}[Hales-Jewett \cite{Hales_Jewett63}]\label{thm_halesjewett}
For each $k,r\in\N$ there exists $HJ(k,r)\in\N$ such that for all $n\geq HJ(k,r)$ and any $r$ coloring of $[k]^n$, there exists a monochromatic combinatorial line.
\end{theorem}

\section{Precise formulations of main results}\label{section_statements}

Recall that, for each triple $(m,p,c)\in\N^3$, an $(m,p,c)$-set is the image of some vector ${\bf s}\in(\N\setminus\{0\})^{m+1}$ under a finite set of semigroup homomorphisms ${\bf x}\mapsto\langle {\bf x},\xi\rangle$ (see Definition \ref{def_intrompc} and explanation right after Definition \ref{def_introdeuber}).
We can generalize this concept by allowing more general classes of mappings.

\begin{definition}\label{def_deuberset}Let $G$ be a countable commutative semigroup.
\begin{enumerate}
\item A \emph{shape} in $G$ is a triple $(m,\F,c)$ where $m\in\N$, $c:G\to G$ is a homomorphism and $\F$ is an $m$-tuple $\F=(F_1,\dots,F_m)$ where each $F_j$ is a finite set of functions from $G^j$ to $G$.
\item  Given a shape $(m,\F,c)$ and ${\bf s}=(s_0,\dots,s_m)\in (G\setminus\{0\})^{m+1}$, the $(m,\F,c)$-set generated by ${\bf s}$ is the set
  $$D(m,\F,c;{\bf s}):=\left\{\begin{array}{lr}c(s_0)&\\f(s_0)+c(s_1),& f\in F_1\\f(s_0,s_1)+c(s_2),& f\in F_2\\ \vdots&\vdots\\f(s_0,\dots,s_{m-1})+c(s_m),& f\in F_m
\end{array}\right\}$$
\end{enumerate}
\end{definition}

To see how the notion of $(m,\F,c)$-sets generalizes the concept of $(m,p,c)$-sets, take a triple $(m,p,c)\in\N^3$.
Let $\tilde c:x\mapsto cx$ where $x\in\Z$ and, for each $j=1,\dots,m$, let $F_j$ be the set of all maps $f:\Z^j\to\Z$ of the form $f:{\bf x}\mapsto\langle {\bf x},{\bf\xi}\rangle$ with ${\bf\xi}\in\{-p,\dots,p\}^j$.
If we take $\F=(F_1,\dots,F_m)$, then for each ${\bf s}\in(\Z\setminus\{0\})^{m+1}$, $D(m,p,c;{\bf s})=D(m,\F,\tilde c;{\bf s})$.

We are interested in \shp s which are partition regular.
\begin{definition}
Let $G$ be a countable commutative semigroup and let $(m,\F,c)$ be a shape in $G$. We say that $(m,\F,c)$ is \emph{partition regular} if for every finite partition $G=C_1\cup\cdots\cup C_r$ there exists $i\in\{1,\dots,r\}$ and ${\bf s}\in(G\setminus\{0\})^{m+1}$ such that $D(m,\F,c;{\bf s})\subset C_i$.
\end{definition}

One could wishfully hope that any shape $(m,\F,c)$ in a countable commutative semigroup $G$ is partition regular.
This, however, is not true in general.
\begin{example}
  Take $G=\N$, partitioned into odd numbers and even numbers, and consider the shape $(1,\F,c)$ where $c$ is the identity map and $\F=(F_1)$ is comprised of the two functions $x\mapsto x$ and $x\mapsto x+1$. Then a $(1,\F,c)$-set is a triple $\{s_0,s_0+s_1,s_0+1 +s_1\}$, but neither the odd numbers nor the even numbers contain a configuration with two consecutive elements.
\end{example}

Another example, when all the maps involved are homomorphisms, is the following.
\begin{example}
  Let $G=\Z^2$, let $m=1$ and let $c:\Z^2\to\Z^2$ be the map $c(x,y)=(x,0)$.
Let $\F=(F_1)$ where $F_1\subset\E(\Z^2)$ consists of the maps $(x,y)\mapsto(0,0)$, $(x,y)\mapsto(0,x)$ and $(x,y)\mapsto(0,y)$.
Finally, consider the partition
$$\Z^2=\{(0,0)\}\cup\{(x,0):x\neq0\}\cup\{(0,y):y\neq0\}\cup\{(x,y):x,y\neq0\}.$$
It is not hard to see that there is no $(m,\F,c)$-set in a single cell of this partition.
\end{example}

In this paper we establish sufficient conditions for a shape $(m,\F,c)$ to be partition regular.
This will allow us to obtain the strong generalizations of Deuber's theorem alluded to in the introduction.

For two (countable commutative) semigroups $G,H$ we denote by $\HH(H,G)$ the set of all semigroup homomorphisms from $H$ to $G$.
We also use $\E(G)$ to denote $\HH(G,G)$ (elements of $\E(G)$ are often referred to as endomorphisms).
Finally, recall that $\PP(H,G)$ denotes the set of polynomial maps $f:H\to G$ with $f(0)=0$ (see Definition \ref{def_polmaps}).

Here is the formulation of one of the main results of this paper; its proof is given at the end of Section \ref{section_mpcincentral}.

\begin{theorem}\label{thm_statements_central}
Let $G$ be a countable commutative semigroup, let $A\subset G$ be a central set and let $(m,\F,c)$ be a shape in $G$. Assume that at least one of the following holds:
\begin{enumerate}
\item The map $c$ is the identity map and, for each $j=1,\dots, m$, we have $F_j\subset\HH(G^j,G)$.
\item $G$ is a group, the image of $c$ has finite index in $G$ and, for each $j=1,\dots, m$, $F_j\subset\PP(G^j,G)$.
\end{enumerate}
Then $A$ contains an $(m,\F,c)$-set.
\end{theorem}
A special case of this theorem was obtained by Furstenberg, who showed that any central set in $\N$ contains a $(m,p,c)$-set for any triple $(m,p,c)\in\N^3$ \cite{Furstenberg81}.

An extension of Furstenberg's result was establish in \cite{Beiglbock_Bergelson_Downarowicz_Fish09}, where it was shown that any $D$-set in $\N$ contains a $(m,p,c)$-set for any triple $(m,p,c)\in\N^3$.
The following theorem strengthens the second part of Theorem \ref{thm_statements_central} in the case\footnote{We believe that Theorem \ref{thm_statements_Dset} is actually valid for general countable commutative groups, but to prove it one would need an appropriate generalization of Theorem \ref{thm_mipsz}, which is currently unavailable.} $G=\Z^n$ ; its proof is presented at the end of Section \ref{section_mpcincentral}.

\begin{theorem}\label{thm_statements_Dset}
Let $A\subset\Z^n$ be a $D$-set and let $(m,\F,c)$ be a \shp{} in $\Z^n$, where the image of $c$ has finite index in $\Z^n$, and for each $j=1,\dots,m$ we have $F_j\subset\PP(\Z^{nj},\Z^n)$. Then $A$ contains an $(m,\F,c)$-set.
\end{theorem}

An immediate corollary of Theorem \ref{thm_statements_central} is that certain rather general types of \shp s are partition regular:
\begin{corollary}\label{thm_statements_main}
Let $G$ be a countable commutative semigroup and let $(m,\F,c)$ be a shape in $G$. Assume that at least one of the following holds:
\begin{enumerate}
\item The map $c$ is the identity map and, for each $j=1,\dots, m$, we have $F_j\subset\HH(G^j,G)$.
\item $G$ is a group, the image of $c$ has finite index in $G$ and, for each $j=1,\dots, m$, $F_j\subset\PP(G^j,G)$.
\end{enumerate}
Then $(m,\F,c)$ is partition regular.
\end{corollary}
We now show how this corollary implies Corollary \ref{cor_intro_multidimbrauer} from the introduction.
\begin{proof}[Proof of Corollary \ref{cor_intro_multidimbrauer}]
  We will use part (1) of Corollary \ref{thm_statements_main}.
  Let $G=\N^d$, $m=1$ and
  $$F_1=\Big\{{\bf x}\mapsto\big(i_1f({\bf x}),\dots,i_df({\bf x})\big):0\leq i_1,\dots,i_d\leq k\Big\}\subset\HH(\N^d,\N^d),$$
  where ${\bf x}$ denotes an element of $\N^d$.
  For any finite coloring of $\N^d$ there is some ${\bf s}=(s_0,s_1)\in\N^2$ such that $D(m,\F,c;{\bf s})$ is monochromatic.
  Putting $b=s_0$ and $a=s_1$ we obtain a monochromatic configuration

  \hfill$\displaystyle\{b\}\cup\Big\{a+\big(i_1f(b),\cdots,i_df(b)\big):0\leq i_1,\dots,i_d\leq k\Big\}.$\hfill\qedhere
\end{proof}

The main tool employed by Furstenberg in his proof of the special case of Theorem \ref{thm_statements_central} mentioned above was his central sets theorem (cf. Theorem \ref{thm_cst}).
A similar strategy was adopted in \cite{Beiglbock_Bergelson_Downarowicz_Fish09} to establish the result for $D$-sets.
Our proof of Theorem \ref{thm_statements_central} is based on the following polynomial version of the central sets theorem, which we believe is of independent interest.

\begin{theorem}[Multidimensional polynomial central sets theorem]\label{thm_PCSTAG}
Let $G$ be a countable abelian group, let $j\in\N$ and let $(y_\alpha)_{\alpha\in{\mathcal F}}$ be an IP-set in $G^j$.
Let $F\subset\PP(G^j,G)$ and let $A\subset G$ be a central set or, if $G=\Z^n$, let $A$ be a $D$-set.
Then there exist an IP-set $(x_\beta)_{\beta\in{\mathcal F}}$ in $G$ and a sub-IP-set $(z_\beta)_{\beta\in{\mathcal F}}$ of $(y_\alpha)_{\alpha\in{\mathcal F}}$ such that
$$\forall f\in F\quad\forall \beta\in{\mathcal F}\qquad x_\beta+f(z_\beta)\in A$$
\end{theorem}
When $G=\Z$ and the polynomial maps in $F$ are homomorphisms, this reduces to the classical central sets theorem. Theorem \ref{thm_PCSTAG} will be derived as a corollary of the more general Theorem \ref{thm_GCST} below.

 As we mentioned in the introduction, one of the main motivations for Deuber to introduce $(m,p,c)$-sets was to prove a conjecture of Rado stating that for a finite partition of rich sets, one of the cells is still rich.

We obtain an analogous result for certain $(m,\F,c)$-sets.
Before we state the main result in this direction (Theorem \ref{thm_lambdacrich} below) we need a few definitions.

\begin{definition}Let $G$ be a countable commutative semigroup.
A \emph{\shd} in $G$ is an infinite (not necessarily countable) set of \shp s.
Given a \shd{} $\Lambda$ in $G$, we say that a set $A\subset G$ is \emph{$\Lambda$-rich} if for every shape $(m,\F,c)\in\Lambda$ there exists an $(m,\F,c)$-set contained in $A$.
\end{definition}

For example, let $\Lambda$ be the \shd{} in $\N$ consisting of the \shp s that arise from all possible triples $(m,p,c)\in\N^3$. Then a set $A\subset\N$ is $\Lambda$-rich if and only if it is rich in the sense defined in the introduction. Here are more examples.
\begin{example}
\begin{enumerate}
\leavevmode
\item\label{enumerate_example1}
Let $k\in\N$, let $c:\N\to\N$ be the identity map, let $F_{1,k}=\{x\mapsto ix:i=0,\dots,k-1\}\subset\E(\N)$ and make $\F_k=(F_{1,k})$.
Then any $(1,\F_k,c)$-set contains a ``Brauer configuration''  of length $k$ (i.e. an arithmetic progression of length $k$ together with its common difference, cf. Theorem \ref{thm_introbrauer}).
\item Let $c$ and $\F_k$ be as in part (\ref{enumerate_example1}) above.
Consider the \shd{} $\Lambda=\{(1,\F_k,c):k\in\N\}$.
A set $A\subset\N$ is $\Lambda$-rich if and only it contains Brauer configurations of arbitrary length.

\item\label{enumerate_example3}
Let again $m\in\N$ and let $c:\N\to\N$ be the identity map.
For each $j=1,\dots,m$, let $F_{j,m}$ be the set of all maps $f:\N^j\to\N$ of the form $f:{\bf x}\mapsto\langle{\bf x},\xi\rangle$ where $\xi\in\{0,1\}^j$.
Let $\F_m=(F_{1,m},\dots,F_{m,m})$.
Then any $(m,\F_m,c)$-set is a set of the form $FS(A)$ for some set $A\subset\N$ with cardinality $m+1$.
\item Let $c$ and $\F_m$ be as in part (\ref{enumerate_example3}) of this example.
Define the \shp{} $\Lambda=\{(m,\F_m,c):m\in\N\}$.
A set $A\subset\N$ is $\Lambda$-rich if and only if it is an $IP_0$ set, i.e. a set containing $FS(A)$ for some arbitrarily large finite sets $A$.
\end{enumerate}
\end{example}

\begin{definition}
Let $G$ be a countable commutative semigroup.
Let $m\in\N$ and, for each $i=1,\dots,m$, let $F_i\subset\HH(G^i,G)$ be finite. Also, let $\F=(F_1,\cdots,F_m)$ and let $c\in\E(G)$. We say that $c$ is \emph{concordant} with $\F$ if there exists a non-zero homomorphism $b\in\E(G)$ and, for each $i\in[m]$ and $f\in F_i$, there is a homomorphism $a_f\in\HH(G^i,G)$ such that $c\circ a_f=f\circ \mathbf{b}$, where $\mathbf{b}:G^i\to G^i$ is the homomorphism $\mathbf{b}(g_1,\dots,g_i)=\big(b(g_1),\dots,b(g_i)\big)$.\end{definition}

Observe that the identity homomorphism $c:x\mapsto x$ is concordant with any $\F$.
More generally, if $c$ is in the center of the semigroup $\E(G)$, then $c$ is concordant with any $\F$ (by taking $b=c$ and $a_f=f$).

When $c$ is an automorphism, it is concordant with any $\F$.
Indeed, one can take $b$ to be the identity map and $a_f=c^{-1}\circ f$.
In the following example, $c$ is neither in the center of $\E(G)$ nor is it an automorphism.
\begin{example}
  Let $G=\Z^2$, let $m=1$, let $c\in\E(\Z^2)$ be the projection onto the first coordinate and let $\F=(F_1)$ where $F_1$ consists of finitely many endomorphisms of $\Z^2$ whose image is contained in $c(\Z^2)$.
  Then $c$ is concordant with $\F$.

  Indeed, take $f\in F_1$.
  We let $b\in\E(\Z^2)$ be the identity map and $a_f=f$.
  Since the restriction of $c$ to its image is the identity map, we have $c\circ a_f=f\circ b$.
\end{example}

One can reinterpret each part of Corollary \ref{thm_statements_main} as providing an example of a \shd{} $\Lambda$ such that, for any finite partition of $G$, one of the cells is $\Lambda$-rich.
Our next theorem provides a natural example of a \shd{} $\Lambda$ with the stronger property that, for any finite partition of a $\Lambda$-rich set, one of the cells is still $\Lambda$-rich.
Theorem \ref{thm_lambdacrich} deals with an arbitrary countable commutative semigroup $G$; when $G=\N$ we recover Deuber's result (Theorem \ref{thm_introdeuber}, part (2)).

\begin{theorem}\label{thm_lambdacrich}
  Let $G$ be a countable commutative semigroup and let $\Lambda_t$ be the \shd{} consisting of all \shp s $(m,\F,c)$ with $m\in\N$, $c$ in the center of $\E(G)$ and $\F=(F_1,\dots,F_m)$ where each $F_j\subset\HH(G^j,G)$. In other words
  $$\Lambda_t=\left\{(m,\F,c):\begin{array}{l}m\in\N,~c\text{ is in the center of }\E(G),\\ \F=(F_1,\dots,F_m),~F_j\subset\HH(G^j,G)~\forall j\end{array}\right\}$$
  For any finite partition of a $\Lambda_t$-rich set, one of the cells is still $\Lambda_t$-rich.
\end{theorem}

If we take $G=\N$ then $\E(\N)$ is isomorphic to the multiplicative semigroup $(\N,\times)$ and hence is commutative; this means that any \shp{} $(m,\F,c)$ arising from a triple $(m,p,c)\in\N^3$ as explained after Definition \ref{def_deuberset}, is in $\Lambda_t$.
Therefore Theorem \ref{thm_lambdacrich} implies Deuber's theorem (Theorem \ref{thm_introdeuber}).
Theorem \ref{thm_lambdacrich} will be derived in Section \ref{section_finitistic} from its finitistic version, which we now state.
\begin{theorem}\label{thm_finitistic}
  Let $G$ be a countable commutative semigroup and let $\Lambda_c$ be the \shd{} of all \shp s $(m,\F,c)$ where $m\in\N$, $F_i\subset\HH(G^i,G)$ for all $i=1,\dots,m$ and $c$ is concordant with $\F$.

  For any $r\in\N$ and any \shp{} $(m,\F,c)\in\Lambda_c$ there exists another \shp{} $(M,\vH,C)\in\Lambda_c$ such that
  any partition of an $(M,\vH,C)$-set into $r$-cells, one of the cells contains an $(m,\F,c)$-set.

  Moreover, if $c$ is the identity, we can take $C$ to be the identity as well, and if $c$ is in the center of $\E(G)$ we can take $C$ to be in the center of $\E(G)$.
\end{theorem}

The proof of Theorem \ref{thm_finitistic} occupies most of Section \ref{section_finitistic}.

The following definition was introduced by Deuber and Hindman in \cite{Deuber_Hindman87}.
For finitely many finite sets $A_1,\dots,A_n$ we define the sum
$$\sum_{i=1}^nA_i:=\{a_1+\cdots+a_n: a_1\in A_1,\dots,a_n\in A_n\}$$

\begin{definition}[$(m,p,c)$-system]
  A set $A\subset\N$ is an \emph{$(m,p,c)$-system} if for each $(m,p,c)\in\N^3$ there exists ${\bf s}={\bf s}(m,p,c)\in\N^{m+1}$ such that
  \begin{enumerate}
    \item $D\big(m_1,p_1,c_1;{\bf s}(m_1,p_1,c_1)\big)\cap D\big(m_2,p_2,c_2;{\bf s}(m_2,p_2,c_2)\big)=\emptyset$ whenever

        $(m_1,p_1,c_1)\neq(m_2,p_2,c_2)$.
    \item For all nonempty finite sets $\alpha\subset\N^3$, we have
    $$\sum_{(m,p,c)\in\alpha}D\big(m,p,c;{\bf s}(m,p,c)\big)\subset A$$
  \end{enumerate}
\end{definition}

In \cite{Deuber_Hindman87} it was proved that for any finite partition of $\N$ one of the cells is an $(m,p,c)$-system, and in \cite{Hindman_Lefmann93}  it was show that for any finite partition of an $(m,p,c)$-system, one of the cells is an $(m,p,c)$-system.
We have extensions of both results for certain countable \shd s in countable commutative semigroups.
\begin{definition}[$\Lambda$-system]Let $G$ be a countable commutative semigroup and let $\Lambda$ be a countable \shd{} in $G$.
A set $A\subset G$ is a $\Lambda$\emph{-system} if for every \shp{} $(m,\F,c)\in\Lambda$ there exists ${\bf s}={\bf s}(m,\F,c)\in G^{m+1}$ such that
  \begin{enumerate}
    \item $D\big(m_1,\F_1,c_1;{\bf s}\big)\cap D\big(m_2,\F_2,c_2;{\bf s}\big)=\emptyset$ whenever $(m_1,\F_1,c_1)\neq(m_2,\F_2,c_2)$.
    \item For all $\alpha\in{\mathcal F}(\Lambda)=\{\alpha\subset\Lambda:0<|\alpha|<\infty\}$ we have
    $$\sum_{(m,\F,c)\in\alpha}D\big(m,\F,c;{\bf s}(m,\F,c)\big)\subset A$$
  \end{enumerate}
\end{definition}

We have two results regarding $\Lambda$-systems.
The first, Theorem \ref{thm_DLambdamonochromatic}, extends the result of \cite{Deuber_Hindman87} and utilizes, in its proof, Corollary \ref{thm_statements_main}.
The second, Theorem \ref{thm_systemspartitionregularity}, extends the result of \cite{Hindman_Lefmann93} and utilizes, in its proof, Theorem \ref{thm_lambdacrich}.
\begin{theorem}\label{thm_DLambdamonochromatic}

Let $G$ and $\Lambda$ satisfy at least one of the following two conditions.
\begin{enumerate}
\item $G$ is a countable commutative semigroup and $\Lambda$ is a countable \shd{} composed by \shp s $(m,\F,c)$ satisfying the first condition of Corollary \ref{thm_statements_main}.

\item $G$ is a countable commutative group and let $\Lambda$ is a countable \shd{} composed by \shp s $(m,\F,c)$ satisfying the second condition of Corollary \ref{thm_statements_main}.
\end{enumerate}
Then for any finite partition of $G$, one of the cells is a $\Lambda$-system.
\end{theorem}

Theorem \ref{thm_DLambdamonochromatic} is proved in Section \ref{sec_mfcsystems}.

Theorem \ref{thm_systemspartitionregularity} below establishes partition regularity of $\Lambda_t$-systems, where the \shd{} $\Lambda_t$ is defined in Theorem \ref{thm_lambdacrich}.
Since $\Lambda_t$-systems do not exist when $\Lambda_t$ is uncountable, we will assume in Theorem \ref{thm_systemspartitionregularity} that the \shd{} $\Lambda_t$ is countable.
Observe that $\Lambda_t$ is countable if and only if $\HH(G^j,G)$ is countable for every $j\in\N$.
Before formulating Theorem \ref{thm_systemspartitionregularity} we provide some relevant examples.

\begin{example}
  If $G$ is a finitely generated commutative semigroup, then the \shd{} $\Lambda_t$ is countable.
  Indeed, for every $j\in\N$ the semigroup $G^j$ is finitely generated, hence a homomorphism $f\in\HH(G^j,G)$ is determined by finitely many values.
  This fact implies that $\HH(G^j,G)$ is countable, and in particular, the center of $\E(G)$ is also countable.
  It follows that $\Lambda_t$ is countable.
\end{example}
\begin{example}
  If $G$ is the additive group of an algebraic number field, then $\Lambda_t$ is countable.
  Indeed, $G$ is isomorphic to $\Q^n$ for some $n\in\N$, and hence, for each $j\in\N$, a homomorphism $f\in\HH(G^j,G)=\HH(\Q^{nj},\Q^n)$ is determined by finitely many points, namely the $nj$ points of the form $(0,\dots,0,1,0,\dots,0)\in\Q^{nj}$.
  This fact implies that $\HH(G^j,G)$ is countable, and in particular the center of $\E(G)$ is also countable.
  It follows that $\Lambda_t$ is countable.
\end{example}

\begin{theorem}\label{thm_systemspartitionregularity}
Let $G$ be a countable commutative semigroup and let $\Lambda_t$ be the \shd{} defined in Theorem \ref{thm_lambdacrich}.
Assume that $\Lambda_t$ is countable.
Then for any finite partition of a $\Lambda_t$-system, one of the cells in the partition is still a $\Lambda_t$-system.
\end{theorem}
Theorem \ref{thm_systemspartitionregularity} is proved in Section \ref{sec_mfcsystems}.

\section{Idempotent ultrafilters and $(m,\F,c)$-sets}\label{section_mpcincentral}
Theorem \ref{thm_statements_Dset} and parts (1) and (2) of Theorem \ref{thm_statements_central} have similar proofs.
To avoid repetition, we unify the three results into a single abstract result; this is Theorem \ref{thm_generalmain} below. Before formulating it, we need to introduce some definitions.

\begin{definition}[R-family]\label{def_famofrec}
Let $G,H$ be countable commutative semigroups and let $p\in\beta G$ be an ultrafilter. Let $\Gamma$ be a set of functions from $H\to G$.
We say that $\Gamma$ is an \emph{R-family\footnote{R stands for \textit{returns}.} with respect to $p$} if for every finite set $F\subset\Gamma$, every $A\in p$ and every IP-set $(y_\alpha)_{\alpha\in{\mathcal F}}$ in $H$, there exist $x\in G$ and $\alpha\in{\mathcal F}$ such that
$$x+f(y_\alpha)\in A\qquad\qquad\forall f\in F$$
\end{definition}

\begin{example}\label{example_essentialFamOfRec}
Let $n,j\in\N$ and take $G=\Z^n$, $H=\Z^{nj}$ and $\Gamma=\PP(\Z^{nj},\Z^n)$.
Then $\Gamma$ is an R-family with respect to any essential idempotent ultrafilter.
Indeed, let $p\in\beta\Z^n$ be an essential idempotent and let $A\in p$.
Then $A$ has positive Banach upper density.
Let $F\subset\Gamma$ be any finite set and let $(y_\alpha)_{\alpha\in{\mathcal F}}$ be an IP-set in $\Z^{nj}$.
Theorem \ref{thm_mipsz} implies that there exist $x\in\Z^n$ and $\alpha\in{\mathcal F}$ such that $x+f(y_\alpha)\in A$ for any $f\in F$, which is precisely the condition for being an R-family.
\end{example}

\begin{example}\label{example_minimalFamOfRec}
Let $G$ be a countable abelian group, let $j\in\N$ and let $H=G^j$. Then the family $\Gamma=\PP(G^j,G)$ is an R-family with respect to any minimal idempotent ultrafilter.
Indeed, let $p\in\beta G$ be a minimal idempotent ultrafilter and let $A\in p$.
By definition, $A$ is a central set, hence a piecewise syndetic set.
Fix a finite set $F\subset\Gamma$ and an IP-set $(y_\alpha)_{\alpha\in{\mathcal F}}$ in $G^j$.
It follows from Theorem \ref{thm_ippolvdW} that there exists $a\in A$ and $\alpha\in{\mathcal F}$ such that $a+f(y_\alpha)\in A$ for all $f\in F$, and hence $\Gamma$ is an R-family.
\end{example}

Yet another family of examples is provided by the following proposition.
\begin{proposition}\label{prop_IPvdWgroups}
Let $G$ be a countable commutative semigroup, let $j\in\N$, let $H=G^j$ and let $\Gamma=\HH(G^j,G)$. Then $\Gamma$ is an R-family with respect to any minimal idempotent ultrafilter $p\in\beta G$.
\end{proposition}

\begin{proof}

We remark that when $G$ is a group, Proposition \ref{prop_IPvdWgroups} follows from Corollary \ref{thm_ippolvdW} (note that homomorphisms are polynomial maps of degree at most $1$).

Let $(y_\alpha)_{\alpha\in{\mathcal F}}$ be an IP-set in $G^j$ and let $A\in p$.
Since $A$ is, in particular, a piecewise syndetic set, there exists a finite set $B\subset G$ such that $A-B:=\{x\in G:\exists a\in A,b\in B: x+b=a\}$ is a thick set.

Let $F\subset\Lambda$ be a finite set and let $n=n(|F|,|B|)$ be the number given by Theorem \ref{thm_halesjewett}.
Since $A-B$ is thick, we can find $g\in G$ such that:
$$\forall(f_1,\dots,f_n)\in F^n \qquad\qquad g+\big(f_1(y_1)+\dots+f_n(y_n)\big)\in A-B$$
We can color $F^n$ with $|B|$ colors by associating $(f_1,\dots,f_n)\in F^n$ with an element $b\in B$ such that $g+b+f_1(y_1)+\dots+f_n(y_n)\in A$.
Apply Theorem \ref{thm_halesjewett} to find a variable word $w\in(F\cup\{*\})^n$ whose corresponding combinatorial line is monochromatic.
Let $b\in B$ be the "color" corresponding to the monochromatic combinatorial line, let $C=\big\{i\in\{1,\dots,n\}:w_i\in F\big\}$, let $\alpha=\{1,\dots,n\}\setminus C$ be the positions of the wild card $*$ in $w$ and let
$$x=g+b+\sum_{i\in C}w_i(y_i)$$
For any $f\in F$ we have
$$g+b+\sum_{i\in C}w_i(y_i)+f\left(\sum_{i\in\alpha}y_i\right)\in A$$
and this can be rewritten as $x+f(y_\alpha)\in A$ for all $f\in F$, which finishes the proof.
\end{proof}

\begin{definition}\label{def_stable}
  Let $G,H$ be countable commutative semigroups and let $\Gamma$ be a set of functions from $H$ to $G$.
  We say that $\Gamma$ is \emph{licit} if for any $f\in\Gamma$ and any $z\in H$, there exists a function $\phi_z\in\Gamma$ such that $f(y+z)=\phi_z(y)+f(z)$.
\end{definition}
\begin{example}\label{example_homolicit} Let $G,H$ be countable commutative semigroups and let $\Gamma\subset\HH(H,G)$.
It is not hard to see that $\Gamma$ is licit.
Indeed, note that for every $f\in\Gamma$ and any $z\in H$ one can take $\phi_z=f$ in the definition.
\end{example}
\begin{example}\label{example_polylicit}
If $G,H$ are countable abelian groups, the set $\Gamma=\PP(H,G)$ is licit. Indeed, for each $f\in\Gamma$ and $z\in H$ one can define $\phi_z(y):=f(y+z)-f(z)$.
        Clearly $\phi(0)=0$.
        For any $h\in H$, we have
\begin{align*}
  \phi_z(y+h)-\phi_z(y)&=f(y+z+h)-f(z)-f(y+z)+f(z)\\&=f\big((y+z)+h\big)-f(y+z)
\end{align*}
        If $f\in\PP(H,G)$ has degree $d$, then $f\big((y+z)+h\big)-f(y+z)$ is a polynomial map of degree at most $d-1$ in the variable $y$ (now both $h$ and $z$ are constants), and hence $\phi_z$ is also a polynomial map of degree at most $d$.

\end{example}

\begin{definition}
  Let $G$ be a countable commutative semigroup. An endomorphism $c\in\E(G)$ is called \emph{IP-regular} if for every IP-set $(x_\alpha)_{\alpha\in{\mathcal F}}$ in $G$ there exists an IP-set $(y_\alpha)_{\alpha\in{\mathcal F}}$ such that $\big(c(y_\alpha)\big)_{\alpha\in{\mathcal F}}$ is a sub-IP-set of $(x_\alpha)_{\alpha\in{\mathcal F}}$ (and in particular $\big(c(y_\alpha)\big)_{\alpha\in{\mathcal F}}$ is itself an IP-set).
\end{definition}

When $G=\Z$, any nontrivial endomorphism $c\in\E(\Z)$ is IP-regular.
It's not hard to see that when $G$ is an arbitrary countable abelian group, any endomorphism whose image has finite index is IP-regular.
We can now formulate our abstract theorem (which has Theorems \ref{thm_statements_central} and \ref{thm_statements_Dset} as corollaries):

\begin{theorem}\label{thm_generalmain}
  Let $G$ be a countable commutative semigroup, let $p\in\beta G$ be an idempotent ultrafilter and let $\Gamma_1,\Gamma_2,\dots$ be R-families with respect to $p$ which are licit, where $\Gamma_j$ consists of maps from $G^j$ to $G$.
  Let $c:G\to G$ be IP-regular, let $m\in\N$ and, for each $j=1,\dots,m$, let $F_j\subset\Gamma_j$ be finite.
  Finally, put $\F=(F_1,\dots,F_m)$.
  Then for any $A\in p$ there exists an IP-set $\big({\bf s}_\alpha\big)_{\alpha\in{\mathcal F}}$ in $G^{m+1}$ such that $D(m,\F,c;{\bf s}_\alpha)\subset A$ for every $\alpha\in{\mathcal F}$.
\end{theorem}

In order to prove Theorem \ref{thm_generalmain} we first need to establish an abstract version of the central sets theorem.

\begin{theorem}\label{thm_GCST}
  Let $G,H$ be countable commutative semigroups, let $p\in\beta G$ be an idempotent ultrafilter, let $\Gamma$ be an R-family with respect to $p$ which is licit. Then for any finite set $F\subset\Gamma$, any $A\in p$ and any IP set $(y_\alpha)_{\alpha\in{\mathcal F}}$ in $H$, there exists a sub-IP-set $(z_\beta)_{\beta\in{\mathcal F}}$ of $(y_\alpha)_{\alpha\in{\mathcal F}}$ and an IP-set $(x_\beta)_{\beta\in{\mathcal F}}$ in $G$ such that
  $$\forall f\in F\qquad\forall\beta\in{\mathcal F}\qquad x_\beta+f(y_\beta)\in A$$
\end{theorem}

\begin{proof}
Let $B=\{n\in A:A-n\in p\}$.
Because $p$ is an idempotent ultrafilter, $B\in p$. Moreover, by Lemma 4.14 in \cite{Hindman_Strauss98}, for any $n\in B$, we have $B-n\in p$.
We will construct sequences $x_1,x_2,\dots$ in $G$ and $\alpha_1<\alpha_2<\cdots$ in ${\mathcal F}$ inductively, so that for each $n$ we have
\begin{equation}\label{eq_thm_centralfull}\forall f\in F\qquad\forall\beta\subset[n],~\beta\neq\emptyset\qquad x_\beta+f(z_\beta)\in B\end{equation}
where $z_\beta=\sum_{i\in\beta}y_{\alpha_i}$.

Since $\Gamma$ is an R-family with respect to $p$, we can find $\alpha_1\in{\mathcal F}$ and $x_1\in G$ such that $x_1+f(y_{\alpha_1})\in B$ for all $f\in F$; in other words we get (\ref{eq_thm_centralfull}) for $n=1$.

Now assume we have found $x_1,\dots,x_n$ in $G$ and $\alpha_1<\dots<\alpha_n$ in ${\mathcal F}$ such that (\ref{eq_thm_centralfull}) is true.
Let
$$C=B\cap\left(\bigcap_{{\emptyset\neq\beta\subset[n]}\atop{f\in F}}B-x_\beta-f(z_\beta)\right)$$
Each of the sets of the intersection is in $p$, and because $p$ is closed under finite intersections, also $C\in p$.
We now take advantage of the fact that $\Gamma$ is licit to find, for each $f\in F$ and each nonempty $\beta\subset[n]$, a map $\phi_\beta^f\in\Gamma$ such that $f(z_\beta+y)=\phi_\beta^f(y)+f(z_\beta)$.
Let $\Phi=F\cup\{\phi_\beta^f:\emptyset\neq\beta\subset[n];f\in F\}$.
We can now use again the fact that $\Gamma$ is an R-family with respect to $p$ and find $x_{n+1}\in G$ and $\alpha_{n+1}>\alpha_n$ in ${\mathcal F}$ such that $x_{n+1}+f(z_{n+1})\in C$ for all $f\in \Phi$, where $z_{n+1}:=y_{\alpha_{n+1}}$.
We claim that (\ref{eq_thm_centralfull}) holds for $n+1$ with these choices, which will complete the induction and finish the proof.

Indeed, let $f\in F$ and let $\beta\subset[n+1]$ be non-empty. If $\beta\subset[n]$, then $x_\beta+f(z_\beta)\in B$ by the induction hypothesis.
If $\beta=\{n+1\}$, then $x_{n+1}+f(z_{n+1})\in C\subset B$ because $F\subset\Phi$.
Otherwise the set defined by $\gamma:=\beta\setminus\{n+1\}\subset[n]$ is nonempty.
Recalling that $x_\beta=x_\gamma+x_{n+1}$ and $z_\beta=z_\gamma+z_{n+1}$, we have
$$x_{n+1}+\phi_\gamma^f(z_{n+1})\in C\subset B-x_\gamma-f(z_\gamma),$$
so
$$x_\gamma+x_{n+1}+f(z_\gamma)+\phi_\gamma^f(z_{n+1})\in B$$
which is equivalent to
$$x_\beta+f(z_\beta)\in B.$$
\qedhere \end{proof}

A concrete corollary of this general result is Theorem \ref{thm_PCSTAG}, which can be interpreted as a polynomial version of the central sets theorem.
It follows from Theorem \ref{thm_GCST} by taking $G$ to be a group and letting $H=G^j$, $\Gamma=\PP(G^j,G)$, and $p$ to be a minimal idempotent (or an essential idempotent if $G=\Z$).
According to Examples \ref{example_essentialFamOfRec} and \ref{example_minimalFamOfRec}, $\Gamma$ is an R-family so Theorem \ref{thm_PCSTAG} follows.

We are now in position to prove Theorem \ref{thm_generalmain}.

\begin{proof}[Proof of Theorem \ref{thm_generalmain}]
What we need to show is that there exists some IP-set $({\bf s}_\alpha)_{\alpha\in{\mathcal F}}$ in $G^{m+1}$ such that, for all $\alpha\in{\mathcal F}$,
\begin{equation}\label{eq_Pdeuber3}\begin{array}{cclcl}
&&c(s_{\alpha,0})&\in&A\\ \forall f\in F_1&\qquad&f(s_{\alpha,0})+c(s_{\alpha,1})&\in&A\\ \forall f\in F_2&\qquad&f(s_{\alpha,0},s_{\alpha,1})+c(s_{\alpha,2})&\in&A\\ \vdots&&\vdots&\vdots&\vdots\\ \forall f\in F_m&\qquad&f(s_{\alpha,0},\dots,s_{\alpha,m-1})+c(s_{\alpha,m})&\in&A\end{array}
\end{equation}

  The proof goes by induction on $m$; assume first that $m=0$.
  Since $A$ belongs to an idempotent ultrafilter, it contains an IP-set, say $(\tilde x_\alpha)_{\alpha\in{\mathcal F}}$.
  Since $c$ is IP-regular, we can find an IP-set $(x_\alpha)_{\alpha\in{\mathcal F}}$ such that $\big(c(x_\alpha)\big)$ is a sub-IP-set of $(\tilde x_\alpha)_{\alpha\in{\mathcal F}}$ and hence $c(x_\alpha)\in A$ for each $\alpha\in{\mathcal F}$.
  Let ${\bf s}^{(0)}_\alpha:=x_\alpha$ for each $\alpha\in{\mathcal F}$.

  Now suppose that $m\geq1$ and we have an IP-set in $G^m$
  $$({\bf s}^{(m-1)}_\alpha)_{\alpha\in{\mathcal F}}=\left(\Big(s^{(m-1)}_{\alpha,0},s^{(m-1)}_{\alpha,1},\dots,s^{(m-1)}_{\alpha,m-1}\Big)\right)_{\alpha\in{\mathcal F}}$$
  such that for any $\alpha\in{\mathcal F}$ we have $D(m-1,\F,c;{\bf s}_\alpha^{(m-1)})\subset A$; in other words, if we take $s_i=s^{(m-1)}_{\alpha,i}$ for each $i=0,\dots,m-1$ we get the first $m$ lines of (\ref{eq_Pdeuber3}), for any $\alpha\in{\mathcal F}$.

  Now apply Theorem \ref{thm_GCST} with $H=G^m$, $\Gamma=\Gamma_m$, $F=F_m$ and $(y_\alpha)_{\alpha\in{\mathcal F}}=({\bf s}^{(m-1)}_\alpha)_{\alpha\in{\mathcal F}}$.
  We obtain a sub-IP-set $({\bf t}_\alpha)$ of $({\bf s}^{(m-1)}_\alpha)$ in $G^m$ and some IP set $(x_\alpha)_{\alpha\in{\mathcal F}}$ in $G$ such that
  \begin{equation}\label{eq_thm_generalmain2}\forall\alpha\in{\mathcal F}\quad\forall f\in F_m \qquad\qquad x_\alpha+f({\bf t}_\alpha)\in A.
  \end{equation}

  Since $c$ is IP-regular we can find an IP-set $(y_\beta)_{\beta\in{\mathcal F}}$ in $G$ such that $\big(c(y_\beta)\big)_{\beta\in{\mathcal F}}$ is a sub-IP-set of $(x_\alpha)_{\alpha\in{\mathcal F}}$; in other words, there exist $\alpha_1<\alpha_2<\cdots$ such that $c(y_\beta)=\sum_{i\in\beta}x_{\alpha_i}$ for all $\beta\in{\mathcal F}$. To ease the notation, let $\alpha_\beta$ denote the set $\alpha_\beta:=\bigcup_{i\in\beta}\alpha_i\in{\mathcal F}$. Then \begin{equation}\label{eq_thm_generalmain3}\forall\beta\in{\mathcal F}\qquad\qquad c(y_\beta)=x_{\alpha_\beta}
  \end{equation}

  Now define $({\bf s}^{(m)}_\beta)_{\beta\in{\mathcal F}}$ by taking the corresponding sub-IP-set of $({\bf t}_\alpha)_{\alpha\in{\mathcal F}}$ for the first $m$ coordinates and letting $(y_\beta)_{\beta\in{\mathcal F}}$ be the last coordinate. More precisely we have:
  $${\bf s}_\beta^{(m)}=\big({\bf t}_{\alpha_\beta},y_\beta\big)\in G^{m+1}$$
  Now fix $\beta\in{\mathcal F}$; we need to show that $D(m,\F,c;{\bf s}^{(m)}_\beta)\subset A$.
  If $j\in\{0,1,\dots,m-1\}$ and $f\in F_j$ then
  \begin{equation}\label{eq_proof_thm_4.9}f\Big(s^{(m)}_{\beta,0},\dots,s^{(m)}_{\beta,j-1}\Big)+c(s^{(m)}_{\beta,j})= f\Big(s^{(m-1)}_{\alpha_\beta,0},\dots,s^{(m-1)}_{\alpha_\beta,j-1}\Big)+c(s^{(m-1)}_{\alpha_\beta,j})\end{equation}
  and the expression in (\ref{eq_proof_thm_4.9}) is in $A$ by induction. If $j=m$ then
\begin{equation}\label{eq_proof_thm_4.9.2}
  f\Big(s^{(m)}_{\beta,0},\dots,s^{(m)}_{\beta,j-1}\Big)+c(s^{(m)}_{\beta,j})= f({\bf t}_{\alpha_\beta})+c(y_\beta)= c(y_\beta)+f({\bf t}_{\alpha_\beta})
\end{equation}
  By (\ref{eq_thm_generalmain3}), the expression in (\ref{eq_proof_thm_4.9.2}) is equal to $x_{\alpha_\beta}+f({\bf t}_{\alpha_\beta})$ and hence, by (\ref{eq_thm_generalmain2}), it is also in $A$. We conclude that $D(m,\F,c;{\bf s}^{(m)}_\beta)\subset A$. This finishes the induction process and the proof.
  \end{proof}

We notice that Theorem \ref{thm_generalmain} allows for repeated terms in $D(m,\F,c;{\bf s})$, in other words, one could have $1\leq i\leq j\leq m$ and $f\in F_i,g\in F_j$ such that
$$f(s_0,\dots,s_{i-1})+c(s_i)=g(s_0,\dots,s_{j-1})+c(s_j)$$
In fact, under the same conditions as Theorem \ref{thm_generalmain}, one may not be able to find ${\bf s}$ for which $D(m,\F,c;{\bf s})$ has no repeated terms.
However, if one makes the additional assumption that for every $j\in\{1,\dots,m\}$ and every $f,g\in F_j$ the set $\big\{{\bf x}\in G^j:f({\bf x})=g({\bf x})\big\}$ is finite, then one can modify the above proof to guarantee the additional property that $D(m,\F,c;{\bf s})$ has no repeated terms.

Indeed, observe that this condition implies that, for every $j\in\{1,\dots,m\}$, the set
$$\big\{{\bf x}\in G^j:(\exists f,g\in F_j):f({\bf x})=g({\bf x})\big\}$$
is finite.
Thus, given any IP-set $({\bf x}_\alpha)_{\alpha\in{\mathcal F}}$ in $G^j$ there exists a sub-IP-set $({\bf y}_\beta)_{\beta\in{\mathcal F}}$ such that for all $\beta\in{\mathcal F}$ and $f,g\in F_j$ one has $f(y_\beta)\neq g(y_\beta)$.
Only one modification of the proof of Theorem \ref{thm_generalmain} is needed to obtain this condition:
after choosing the sub-IP-set $({\bf t}_\alpha)$ of $({\bf s}_\alpha^{(m-1)})$ with the property (\ref{eq_thm_generalmain2}), pass to a further sub-IP-set $({\bf y}_\beta)$ of $({\bf t}_\alpha)$ with the property that for all $\beta\in{\mathcal F}$ and all $f,g\in F_j$ one has $f(y_\beta)\neq g(y_\beta)$.

The following theorem summarizes the above discussion.
\begin{theorem}
  Let $G,p,c,m,F_1,\dots,F_m,\F$ be as in Theorem \ref{thm_generalmain}.
  Assume that for every $j\in\{1,\dots,m\}$ and every $f,g\in F_j$ the set $\big\{{\bf x}\in G^j:f({\bf x})=g({\bf x})\big\}$ is finite.
  Then for any $A\in p$ there exists an IP-set $\big({\bf s}_\alpha\big)_{\alpha\in{\mathcal F}}$ in $G^{m+1}$ such that $D(m,\F,c;{\bf s}_\alpha)$ is contained in $A$ and has no repeated terms, in the sense that for every $\alpha\in{\mathcal F}$ and for all $i,j$ with $1\leq i\leq j\leq m$ and $f\in F_i,g\in F_j$ we have
$$f(s_0,\dots,s_{i-1})+c(s_i)\neq g(s_0,\dots,s_{j-1})+c(s_j)$$
\end{theorem}

We will now deduce Theorems \ref{thm_statements_central} and \ref{thm_statements_Dset} from our abstract Theorem \ref{thm_generalmain}.

\begin{proof}[Proof of Theorem \ref{thm_statements_central}]
Let $G$ be a countable commutative group and let $A\subset G$ be a central set.
Thus, there exists a minimal idempotent $p\in\beta G$ with $A\in p$.

We start by proving part (1).
Assume $(m,\F,c)$ is a \shp{} in $G$ where $c$ is the identity map and that $F_j\subset\HH(G^j,G)$ for each $j=1,\dots,m$.
The endomorphism $c$ is trivially IP-regular.
For each $j\in\N$ let $\Gamma_j=\HH(G^j,G)$; it follows from Proposition \ref{prop_IPvdWgroups} that each $\Gamma_j$ is an R-family with respect to $p$.
Finally, by Example \ref{example_homolicit} each $\Gamma_j$ is licit.
We can now apply Theorem \ref{thm_generalmain} to find ${\bf s}\in G^{m+1}$ with $D(m,\F,c;{\bf s})\subset A$ as desired.

Next we prove part (2).
Assume $G$ is a group and $(m,\F,c)$ is a \shp{} in $G$ where $c$ is an endomorphism whose image has finite index in $G$ and, for each $j=1,\dots,m$, $F_j\subset\PP(G^j,G)$.
To see that $c$ is IP-regular, observe that any IP-set has a sub-IP-set contained in the image of $c$, and that IP-sets carry through homomorphisms.
For each $j\in\N$ let $\Gamma_j=\PP(G^j,G)$; it follows from Example \ref{example_minimalFamOfRec} that each $\Gamma_j$ is an R-family with respect to $p$.
By Example \ref{example_polylicit} each $\Gamma_j$ is licit.
We can now apply Theorem \ref{thm_generalmain} to find ${\bf s}\in G^{m+1}$ with $D(m,\F,c;{\bf s})\subset A$ as desired.

\end{proof}

\begin{proof}[Proof of Theorem \ref{thm_statements_Dset}]
Let $n\in\N$, let $G=\Z^n$ and let $A\subset\Z^n$ be a $D$-set.
Thus, there exists an essential idempotent $p\in\beta(\Z^n)$ with $A\in p$.
Assume $(m,\F,c)$ is a \shp{} in $\Z^n$ where $c$ is an endomorphism whose image has finite index in $\Z^n$ and, for each $j=1,\dots,m$, $F_j\subset\PP(\Z^{nj},\Z^n)$.
To see that $c$ is IP-regular, observe that any IP-set has a sub-IP-set contained in the image of $c$, and that IP-sets carry through homomorphisms.

For each $j\in\N$ let $\Gamma_j=\PP(\Z^{nj},\Z^n)$; it follows from Example \ref{example_essentialFamOfRec} that each $\Gamma_j$ is an R-family with respect to $p$.
By Example \ref{example_polylicit} each $\Gamma_j$ is licit.
We can now apply Theorem \ref{thm_generalmain} to find ${\bf s}\in G^{m+1}$ with $D(m,\F,c;{\bf s})\subset A$ as desired.
\end{proof}

\section{Proofs of Theorems \ref{thm_lambdacrich} and \ref{thm_finitistic}}\label{section_finitistic}
The main purpose of this section is to prove Theorems \ref{thm_lambdacrich} and \ref{thm_finitistic}.

Our proof of Theorem \ref{thm_finitistic} is inspired by a proof of Deuber's original result presented in \cite{Gunderson02}.
Before we start with the proofs we need a definition.

\begin{definition}
Let $G$ be a countable commutative semigroup, let $(m,\F,c)$ be a \shp{} in $G$, let ${\bf s}\in G^{m+1}$ and let $k\in\{0,1,\dots,m\}$.
The \emph{$k$-th line} of the $(m,\F,c)$-set $D(m,\F,c;{\bf s})$ is the set
$$
\{f(s_0,\dots,s_{k-1})+c(s_k):f\in F_k\}$$
\end{definition}
Observe that $D(m,\F,c;{\bf s})$ is the union of its $m+1$ lines.

The proof of Theorem \ref{thm_finitistic} goes by induction.
Due to its complicated nature it is convenient to isolate the induction step as a separate lemma.

\begin{lemma}\label{lemma_induction}
  Let $G$ be a countable commutative semigroup with identity $0$, let $\Lambda_c$ be the \shd{} defined in Theorem \ref{thm_finitistic}, let $(m,\F,c)\in\Lambda_c$ and let $r\in\N$. Then there exists a \shp{} $(M,\vH,C)\in\Lambda_c$ such that for any $r$-coloring of an $(M,\vH,C)$-set such that the last $k$ lines are each monochromatic (but different lines can have different colors) there exists a subset which is an $(m,\F,c)$-set whose last $k+1$ lines are each monochromatic.

  Moreover, if $c$ is the identity map, we can take $C$ to be the identity map as well, and if $c$ is in the center of $\E(G)$ we can take $C$ to be in the center of $\E(G)$.
\end{lemma}
\begin{proof}
 Since any subset of a monochromatic set is monochromatic, we can work with conveniently chosen supersets of the $F_i$'s.
   Hence we may and will assume that each $F_i$ contains the projection homomorphisms  $\pi_j:G^i\to G$ (in each coordinate) and the zero homomorphism.
   We will also add to each $F_i$ all the homomorphisms of the form
   $$\phi(x_0,\dots,x_{i-1})=f(x_0,\dots,x_{j-1}) \qquad\text{with }f\in F_j\text{ and }j<i$$
   The main technical tool of our proof is Hales-Jewett's theorem (Theorem \ref{thm_halesjewett}).
   Let $n=HJ(|F_{m-k}|,r)$ be such that any $r$-coloring of $F_{m-k}^n$ contains a monochromatic combinatorial line.
   Since $c$ is concordant with $\F$, there exists an endomorphism $b:G\to G$ and, for each $f\in F_{m-k}$, there exists $a_f\in\HH(G^{m-k},G)$ such that $c\circ a_f=f\circ\mathbf{b}$ (where $\mathbf{b}\in\E(G^{m-k})$ is defined by $\mathbf{b}(x_1,\dots,x_{m-k})=b(x_1)+\cdots+b(x_{m-k})$).

   For convenience we denote by $N$ the product $N=n(m-k)$ and let $M=N+k$.
   For each $j=1,\dots,M$, let $H_j$ be a finite set of homomorphisms from $G^j\to G$ that will be determined later.
   Let $H_N$ be the set of all homomorphisms $\phi:G^N\to G$ of the form
   \begin{equation*}\phi(t_0,\dots,t_{N-1})=\sum_{i=0}^{n-1}f_i\circ \mathbf{b}(t_{i(m-k)},t_{i(m-k)+1},\dots,t_{i(m-k)+m-k-1})\end{equation*}
   with $f_0,\dots,f_{n-1}\in F_{m-k}$.
   Finally, make $\vH=(H_1,\dots,H_M)$ and $C=c\circ b$.
   Observe that if $c$ is in the center of $\E(G)$, then $b=c$, and hence $C$ is also in the center of $\E(G)$.
   Moreover, if $c$ is the identity map, then $b$ is also the identity map, and so is $C$.

   Let $t_0,\dots,t_M\in G$ be arbitrary and let $S_H$ be the $(M,\vH,C)$-set they induce.
   It will simplify considerably the notation to let $$T_i:=(t_{i(m-k)},t_{i(m-k)+1},\dots,t_{i(m-k)+m-k-1})\in G^{m-k}$$ for each $i=0,\dots,n-1$. Thus, in particular, we can write
   $$H_N=\left\{\phi:(t_0,\dots,t_{N-1})\mapsto\sum_{i=0}^{n-1}f_i\circ\mathbf{b}(T_i):f_0,\dots,f_{n-1}\in F_{m-k}\right\}$$
   Assume that we are given a coloring of $S_H$ into $r$ colors such that each of the last $k$ lines are monochromatic (but not necessarily of the same color).

   Color $w=(f_0,\dots,f_{n-1})\in F_{m-k}^n$ with the color of
   \begin{equation}\label{eq_color}\sum_{i=0}^{n-1}f_i\circ\mathbf{b}(T_i)+C(t_N)\end{equation}
   Observe that the elements in (\ref{eq_color}) are in the $N$th line of $S_H$.
   It follows from the Hales-Jewett theorem that one can find a variable word $w\in(F_{m-k}\cup\{*\})^n$  which induces a monochromatic combinatorial line.
   We let $\langle n\rangle=\{0,\dots,n-1\}$, let $A=\{i\in\langle n\rangle:w_i=*\}$ and let $B=\langle n\rangle\setminus A$.
   Now define
   \begin{equation}\label{eq_udef}u_j=\left\{\begin{array}{rcl}\displaystyle b(t_{M-m+j}) &\text{if}&m-k<j\leq m\\ \displaystyle\sum_{i\in B}a_{w_i}(T_i)+b(t_N)&\text{if}&j=m-k\\ \displaystyle\sum_{i\in A}b(t_{i(m-k)+j})  &\text{if}&0\leq j<m-k\end{array}\right.\end{equation}

Note that, for each $\ell=0,\dots,m$, the point $u_{m-\ell}$ depend only on $t_0,\dots,t_{M-\ell}$.

We claim that, with the right choice of $\vH$, the $(m,\F,c)$-set $S_F$ generated by $u_0,\dots,u_m$ is a subset of $S_H$ and that each of the last $k+1$ lines of $S_F$ are monochromatic.
Indeed, for $m-k<j\leq m$, the $j$-th line of $S_F$ is the set
$$\{f(u_0,\dots,u_{j-1})+c(u_j):f\in F_j\}=\{f(u_0,\dots,u_{j-1})+C(t_{M-m+j}):f\in F_j\}$$
This will be a subset of the line $M-m+j$ of $S_H$ if we make $H_{M-m+j}$ contain all the homomorphisms $\phi$ of the form
$$\phi(t_0,\dots,t_{M-m+j-1})=f(u_0,\dots,u_{j-1})$$
for any $f\in F_j$, any possible choice of $A,B\subset\langle n\rangle$ and any $w_i\in F_{m-k}$ (with the $u_j$'s being determined by (\ref{eq_udef})).
Hence the $j$-th line of $S_F$ is monochromatic.

The $(m-k)$-th line of $S_F$ is the set
\begin{eqnarray*}
  &&\{f(u_0,\dots,u_{m-k-1})+c(u_{m-k}):f\in F_{m-k}\}\\&=&
\left\{f\left(\sum_{i\in A}\mathbf{b}(T_i)\right)+\sum_{i\in B}c\circ a_{w_i}(T_i)+C(t_N):f\in F_{m-k}\right\}\\&=& \left\{\sum_{i\in A}f\circ\mathbf{b}(T_i)+\sum_{i\in B}w_i\circ\mathbf{b}(T_i)+C(t_N):f\in F_{m-k}\right\}
\end{eqnarray*}
which is precisely the monochromatic combinatorial line found by applying the Hales-Jewett's theorem. Hence the $(m-k)$-th line of $S_F$ is inside $S_H$ and it is monochromatic.

For $j<m-k$, the $j$-th line of $S_F$ is the set
\begin{align*}&\{f(u_0,\dots,u_{j-1})+c(u_j):f\in F_j\}\\=&\left\{f(u_0,\dots,u_{j-1})+c\left(\sum_{i\in A}b(t_{i(m-k)+j})\right):f\in F_j\right\}\end{align*}
Let $a=\max A$.
Then the $j$-th line of $S_F$ can be written as
$$\left\{f(u_0,\dots,u_{j-1})+\sum_{i\in A\setminus\{a\}}C(t_{i(m-k)+j})+C(t_{a(m-k)+j}):f\in F_j\right\}$$
which will be contained in the $a(m-k)+j$-th line of $S_H$ if we make $H_{a(m-k)+j}$ contain all the homomorphisms $\phi$ of the form
$$\phi(t_0,\dots,t_{a(m-k)+j})=\sum_{i\in A}f(u_0,\dots,u_{j-1})+C(t_{i(m-k)+j})$$
for any $f\in F_j$ and any possible choice of $A\subset\langle a\rangle$, where the dependence of $u_i$ on $t_i$ is given by (\ref{eq_udef}).

It is routine to verify that $C$ is concordant with $H$.
This finishes the proof.
\end{proof}

We move now to proving Theorem \ref{thm_finitistic}.

\begin{proof}[Proof of Theorem \ref{thm_finitistic}]
If $r=1$ there is nothing to prove so we assume $r>1$.
Let $(m,\F,c)\in\Lambda_c$, let $n=m(r-1)$ and, for each $j=1,\dots,n$, let $H^{(0)}_j$ be a finite set of homomorphisms $\phi:G^j\to G$ of the following form. Take $\ell\in\{1,\dots,j\}$ and let $0\leq i_1<\cdots<i_\ell<j$ be arbitrary.
 Let $f\in F_\ell$ and define
 $$\phi_{f,i_1,\dots,i_\ell}(x_0,\dots,x_{j-1})=f(x_{i_1},x_{i_2},\dots,x_{i_\ell})$$
 We let $H^{(0)}_j=\Big\{\phi_{f,i_1,\dots,i_\ell}\Big|\ell\in\{1,\dots,j\},f\in F_\ell,0\leq i_1<\cdots<i_\ell<j\Big\}$.
 Now let $\vH^{(0)}=(H_1^{(0)},\dots,H^{(0)}_{m_0})$.
 Finally, put $c_0=c$ and $m_0=n$.

Applying repeatedly Lemma \ref{lemma_induction}, we construct inductively sequences
$(m_i)_{i=0}^n$, $\big(\vH^{(i)}\big)_{i=0}^n$ and $(c_i)_{i=0}^n$, such that the \shp{} $(m_i,\vH^{(i)},c_i)$ satisfies the conclusion of Lemma \ref{lemma_induction} when we input the \shp{} $(m_{i-1},\vH^{(i-1)},c_{i-1})$ and set $k=n-i$.

 Let $M=m_n$, $\vH=\vH^{(n)}$ and $C=c_n$.
 By construction, for any $r$-coloring of an $(M,\vH,C)$-set $S_H$ we can find a subset which is a $(m_{n-1},\vH^{(n-1)},c_{n-1})$-set with the last line monochromatic.
 Iterating, we obtain for each $i=0,\dots,n$, a sub $(m_i,\vH^{(i)},c_i)$-set with the last $n-i$ lines monochromatic.
 In particular, setting $i=0$ we obtain a $(n,\vH^{(0)},c)$-set with each line monochromatic (but different lines can have different colors).

 Let ${\bf t}=(t_0,\dots,t_n)$ be the generator of this $(n,\vH^{(0)},c)$-set.
 Applying the pigeonhole principle one can find, among the $n+1$ lines of $D(n,\vH^{(0)},c;{\bf t})$, $m+1$ lines of the same color, say the lines $\ell_0,\ell_2,\dots,\ell_m$.
 For each $j=0,\dots,m$ let $s_j=t_{\ell_j}$ and let ${\bf s}=(s_0,\dots,s_m)$.
 By the construction of $H^{(0)}_j$ we deduce that the $j$-th line of $D(m,\F,c;{\bf s})$ is contained in $\ell_j$-th line of $D(n,\vH^{(0)},c;{\bf t})$.
 Therefore $D(m,\F,c;{\bf s})$ is monochromatic as desired.
\end{proof}

To derive Theorem \ref{thm_lambdacrich} from Theorem \ref{thm_finitistic} we need first to establish a lemma.

\begin{definition}
Given two \shp s $\lambda_1=(m_1,\F^{(1)},c_1)$ and $\lambda_2=(m_2,\F^{(2)},c_2)$ in a countable commutative semigroup $G$, we say that \emph{$\lambda_1$ contains $\lambda_2$} if for every ${\bf s}_1\in G^{m_1+1}$ there exists ${\bf s}_2\in G^{m_2+1}$ such that $$D(m_1,\F^{(1)},c_1;{\bf s}_1)\supset D(m_2,\F^{(2)},c_2;{\bf s}_2).$$
\end{definition}
\begin{lemma}\label{lemma_dsyscontained}
Let $G$ be a countable commutative semigroup and let $\Lambda_t$ be the \shd{} defined in Theorem \ref{thm_lambdacrich}.
For any two \shp s $\lambda_1,\lambda_2\in\Lambda_t$ there exists some \shp{} $\lambda\in\Lambda$ which contains both $\lambda_1$ and $\lambda_2$.
\end{lemma}

\begin{proof}
Let $(m_i,F^{(i)},c_i)=\lambda_i$ for $i=1,2$.
Let $c=c_1\circ c_2=c_2\circ c_1$ and let $m=\max(m_1,m_2)$.
We can assume that $m_1=m_2=m$, putting $F^{(i)}_k=\emptyset$ for $k>m_i$ if necessary.
For each $i=1,2$ and $n=1,\dots,m$, let ${\bf c}_i\in\E(G^n)$ be the map ${\bf c}_i:(g_0,\dots,g_{n-1})\mapsto\big(c_i(g_0),\dots,c_i(g_{n-1})\big)$ and let
$$F_n=\Big\{f\circ {\bf c}_2:f\in F_n^{(1)}\Big\}\cup\Big\{f\circ {\bf c}_1:f\in F_n^{(2)}\Big\}$$
Let $F=(F_1,\dots,F_m)$ and let $\lambda=(m,\F,c)$.
Since both $c_1$ and $c_2$ are in the center of $\E(G)$, so is $c$ and hence $\lambda\in\Lambda_t$.

Finally, given any ${\bf s}\in G^m$ we need to show that $D(m,\F,c;{\bf s})$ contains an $(m_i,F^{(i)},c_i)$-set for each $i=1,2$.
Let ${\bf s}^{(1)}={\bf c}_2({\bf s})=\big(c_2(s_0),\dots,c_2(s_m)\big)$ and ${\bf s}^{(2)}={\bf c}_1({\bf s})=\big(c_1(s_0),c_1(s_1),\dots,c_1(s_m)\big)$.
We claim that
$$D\big(m_i,F^{(i)},c_i;{\bf s}^{(i)}\big)\subset D(m,\F,c;{\bf s})$$
Indeed, for any $i=1,2$, any $n=0,1,\dots,m$ and any $f\in F_n^{(i)}$ we have
\begin{equation}\label{eq_lemma_dsyscontained}c_i(s_n^{(i)})+f\big(s_{n-1}^{(i)},\dots,s_0^{(i)}\big)=c_i\big(c_{3-i}(s_n)\big)+f\big({\bf c}_{3-i}(s_{n-1},\dots,s_0)\big)\end{equation}
Since $c_i\circ c_{3-i}=c$ and for $f\in F_n^{(i)}$ we have $f\circ{\bf c}_{3-i}\in F_n$, we deduce that the element $(\ref{eq_lemma_dsyscontained})$ is in $D(m,\F,c;{\bf s})$ as desired.
\end{proof}

\begin{proof}[Proof of Theorem \ref{thm_lambdacrich}]

Let $A$ be a $\Lambda_t$-rich set and consider an arbitrary finite partition $A=A_1\cup\cdots\cup A_r$.
Assume none of the $A_i$ is $\Lambda$-large.
Then for each $i\in\{1,\dots,r\}$ there exists a \shp{} $\lambda_i\in\Lambda_t$ such that $A_i$ does not contain an $(m,\F,c)$-set of \shp{} $\lambda_i$.

Applying Lemma \ref{lemma_dsyscontained} $r-1$ times, one can find a \shp{} $\lambda\in\Lambda_t$ that contains each of the \shp s $\lambda_1,\dots,\lambda_r$.
Therefore, none of the $A_i$ can contain an $(m,\F,c)$-set of \shp{} $\lambda$.

It follows from Theorem \ref{thm_finitistic} that there exists a shape $(M,\vH,C)\in\Lambda_t$ such that any partition of an $(M,\vH,C)$-set into $r$ cells contains a $(m,\F,c)$-set in a single cell.
On the one hand, because $A$ was assumed to be $\Lambda_t$-large, it will contain an $(M,\vH,C)$-set.
On the other hand, this implies that some $A_i$ contains an $(m,\F,c)$-set, contradicting the construction above.
This contradiction implies that some $A_i$ must be $\Lambda_t$-large.
\end{proof}

Theorems \ref{thm_finitistic} and \ref{thm_lambdacrich} deal only with \shp s $(m,\F,c)$ where each component $F_i$ of $\F$ is a set of homomorphisms.
It is not clear if the methods used to prove them can be adapted to more general \shd s, such as those where the $F_i$ are allowed to contain polynomial maps.

\section{$\Lambda$-Systems}\label{sec_mfcsystems}

In this section we prove Theorems \ref{thm_DLambdamonochromatic} and \ref{thm_systemspartitionregularity} concerning the partition regularity of $\Lambda$-systems.

We start with the proof of Theorem \ref{thm_DLambdamonochromatic}.

\begin{proof}[Proof of Theorem \ref{thm_DLambdamonochromatic}]
We show that in fact any central set contains a $\Lambda$-system.
Let $A$ be a central set and let $p\in\beta G$ be a minimal idempotent such that $A\in p$.
  Let $B=\{n\in A:A-n\in p\}$.
  Observe that $B\in p$ and that $B-n\in p$ for every $n\in B$.

  Next, enumerate $\Lambda=\{\lambda_1,\lambda_2,\dots\}$ and let $\lambda_i=(m_i,\F^{(i)},c_i)$.
  It follows from Theorem \ref{thm_generalmain} that there exists ${\bf s}_1\in G^{m_1+1}$ such that $D(m_1,\F^{(1)},c_1;{\bf s}_1)\subset B$.

  We will construct inductively a sequence ${\bf s}_1,{\bf s}_2,\dots$ such that for all $n\in\N$ ${\bf s}_n\in G^{m_n+1}$ and such that
  \begin{equation}\label{eq_thm_dsystemsgeneral1}\forall \alpha\subset[n], \alpha\neq\emptyset\qquad\qquad\sum_{i\in\alpha}D(m_i,\F^{(i)},c_i;{\bf s}_i)\subset B\end{equation}
  Above we found ${\bf s}_1$ such that (\ref{eq_thm_dsystemsgeneral1}) holds with $n=1$.
  Assume now that ${\bf s}_1,\dots,{\bf s}_n$ satisfying (\ref{eq_thm_dsystemsgeneral1}) have been found.
  Let
  $$T_n=\{0_G\}\cup\bigcup_{\emptyset\neq\alpha\subset[n]}\sum_{i\in\alpha}D(m_i,\F^{(i)},c_i;{\bf s}_i)$$
  and let $B_n=\bigcap_{x\in T_n}(B-x)\setminus T_n$.
  Observe that $B_n\subset B$.
  Since for each $x\in T_n$ we have $B-x\in p$ and because $p$ is an ultrafilter and hence closed under finite intersections, we deduce that $B_n\in p$. (removing the finite set $T_n$ does not affect this because $p$ is not principal and hence can not contain finite sets.)

  Using Theorem \ref{thm_generalmain} again we can find some ${\bf s}_{n+1}\in G^{m_{n+1}+1}$ such that $$D(m_{n+1},\F^{(n+1)},c_{n+1};{\bf s}_{n+1})\subset B_n.$$
  We claim that for this choice of ${\bf s}_{n+1}$ the inclusions (\ref{eq_thm_dsystemsgeneral1}) hold with $n+1$.

  Indeed, if $\emptyset\neq\alpha\subset[n+1]$ does not contain $n+1$, then (\ref{eq_thm_dsystemsgeneral1}) follows by induction.
  If $n+1\in\alpha$ then let $\beta=\alpha\setminus\{n+1\}\subset[n]$ and let $Q=\sum_{i\in\beta}D(m_i,\F^{(i)},c_i;{\bf s}_i)$.
  Observe that
  $$D(m_{n+1},\F^{(n+1)},c_{n+1};{\bf s}_{n+1})\subset B_n\subset\bigcap_{x\in Q}(B-x)$$
  Thus $\sum_{i\in \alpha}D(m_i,\F^{(i)},c_i;{\bf s}_i)=D(m_{n+1},\F^{(n+1)},c_{n+1};{\bf s}_{n+1})+Q\subset B$.
  This proves the claim that (\ref{eq_thm_dsystemsgeneral1}) holds for $n+1$, which finishes the induction.
  \end{proof}

To prove Theorem \ref{thm_systemspartitionregularity} takes some more work.

\begin{lemma}\label{lemma_cofinite}
Let $G$ and $\Lambda_t$ be as in Theorem \ref{thm_systemspartitionregularity} and let $\tilde\Lambda\subset\Lambda_t$ be a cofinite subset. Then any $\tilde\Lambda$-rich set is a $\Lambda_t$-rich set.
\end{lemma}
\begin{proof}
Since $\tilde\Lambda$ is cofinite in $\Lambda_t$ there must exist some \shp{} $\tilde\lambda\in\Lambda_t$ which is not contained in any \shp{} in $\Lambda_t\setminus\tilde\Lambda$; in other words, any \shp{} containing $\tilde\lambda$ is in $\tilde\Lambda$.

For any \shp{} $\lambda\in\Lambda_t$ one can use Lemma \ref{lemma_dsyscontained} to find another \shp{} $\phi(\lambda)\in\Lambda_t$ which contains both $\lambda$ and $\tilde\lambda$. Thus $\phi(\lambda)$ is actually inside $\tilde\Lambda$.

Given any $\tilde\Lambda$-rich set $A$, it contains a $\phi(\lambda)$-set for each $\lambda\in\Lambda_t$, hence $A$ contains a $\lambda$-set for each $\lambda\in\Lambda_t$, which is to say, $A$ is $\Lambda_t$-rich.
\end{proof}

\begin{lemma}\label{lemma_dsystemsinidempotents}
Let $G$ and $\Lambda_t$ be as in Theorem \ref{thm_systemspartitionregularity} and let $U\subset\beta G$ be the set of ultrafilters such that for every $p\in U$, any element $A\in p$ is $\Lambda_t$-rich.
Then $U$ contains a non-empty compact semigroup.
Moreover a set $A\subset G$ is a $\Lambda_t$-system if and only if there exists an idempotent ultrafilter $p\in U$ such that $A\in p$.
\end{lemma}

\begin{proof}
First assume that $A\subset G$ is a $\Lambda_t$-system.
Let ${\bf s}$ be a function which assigns to each \shp{} $\lambda=(m,\F,c)\in\Lambda_t$ a vector ${\bf s}(\lambda)\in G^{m+1}$ such that
$$\bigcup_{\alpha\in{\mathcal F}(\Lambda_t)}\left(\sum_{\lambda\in\alpha}D\big(\lambda;{\bf s}(\lambda)\big)\right)\subset A$$
For each subset $\tilde\Lambda\subset\Lambda_t$ denote by $DS(\tilde\Lambda)$ the set
$$DS(\tilde\Lambda)=\bigcup_{\alpha\in{\mathcal F}(\tilde\Lambda)}\left(\sum_{\lambda\in\alpha}D\big(\lambda;{\bf s}(\lambda)\big)\right)$$
and let $\tilde K\subset\beta G$ be the intersection of the compact sets $\overline{DS(\tilde\Lambda)}$ as $\tilde\Lambda$ runs over all cofinite subsets of $\Lambda_t$. In other words
$$\tilde K=\bigcap_{\alpha\in{\mathcal F}(\Lambda_t)}\overline{DS(\Lambda_t\setminus\alpha)}\quad\subset\beta G$$
Next let $K=\tilde K\cap U$; we claim that $K$ is non-empty.
First, note that for any cofinite $\tilde\Lambda\subset\Lambda_t$, the set $DS(\tilde\Lambda)$ is $\tilde\Lambda$-rich, and hence, in view of Lemma \ref{lemma_cofinite}, it is also $\Lambda_t$-rich.
Using Theorem \ref{thm_lambdacrich}, it follows from \cite[Theorem 3.11]{Hindman_Strauss98} that the intersection $U\cap\overline{DS(\tilde\Lambda)}$ is a non-empty compact set.
Finally, for any finitely many cofinite subsets $\Lambda_1,\dots,\Lambda_k$ of $\Lambda_t$, we have that $\tilde\Lambda=\Lambda_1\cap\cdots\cap\Lambda_k$ is itself a cofinite subset of $\Lambda_t$, so the intersection
$$\bigcap_{i=1}^k\big(\overline{DS(\Lambda_i)}\cap U\big)\supset\overline{DS(\tilde\Lambda)}\cap U$$
is nonempty.
This implies that the infinite intersection $K=\tilde K\cap U$ is also nonempty, proving the claim.

Observe that for any $p\in K$, since $DS(\tilde\Lambda)\subset DS(\lambda)\subset A$ for any subset $\tilde\Lambda\subset\Lambda$, we have that $A\in p$.
Our strategy now is to show that $K$ is a closed semigroup and with the help of Ellis's lemma\footnote{Ellis's lemma \cite{Ellis58} states that any semi-continuous compact semigroup contains an idempotent.} find an idempotent $p\in K$.
This idempotent will be an element of $U$ that contains $A$.

Let $p,q\in K$; we first show that $p+q\in U$.
Indeed let $B\in p+q$ and let $\lambda=(m,\F,c)\in\Lambda_t$ be an arbitrary \shp.
By definition, $\{n\in G:B-n\in p\}\in q$.
Let ${\bf s}^{(1)}\in G^m$ be such that $D(\lambda;{\bf s}^{(1)})\subset\{n\in G:B-n\in p\}$, let
$$C=\bigcap_{n\in D(\lambda;{\bf s}^{(1)})}B-n\qquad\in p$$
and let ${\bf s}^{(2)}\in G^m$ be such that $D(\lambda;{\bf s}^{(2)})\subset C$.
Thus, in particular, for any $j=0,1,\dots,m$ and any $f\in F_j$ we have
$$\Big(c(s_j^{(2)})+f(s_{j-1}^{(2)},\dots,s_0^{(2)})\Big)+c(s_j^{(1)})+f(s_{j-1}^{(1)},\dots,s_0^{(1)}\in B$$
Thus, taking ${\bf s}={\bf s}^{(1)}+{\bf s}^{(2)}$, and because $c,f$ are homomorphisms, we have $D(\lambda;{\bf s})\subset B$. Since $\lambda\in\Lambda_t$ was arbitrary we conclude that $B$ is $\Lambda_t$-large, and because $B\in p+q$ was arbitrary, we conclude that $p+q\in U$.

Next we need to show that for any cofinite $\tilde\Lambda\subset\Lambda_t$ we have $p+q\in\overline{DS(\tilde\Lambda)}$, which is equivalent to $DS(\tilde\Lambda)\in p+q$.
By definition, this is equivalent to $\{n\in G:DS(\tilde\Lambda)-n\in p\}\in q$ and since both $p,q\in K$, this will follow if we show that
\begin{equation}\label{eq_lemma_dsystemsinidempotents}DS(\tilde\Lambda)\subset\{n\in G:DS(\tilde\Lambda)-n\in p\}\end{equation}
Fix $n\in DS(\tilde\Lambda)$.
Then we can decompose $n=\sum_{\lambda\in\alpha}x_\lambda$ for some $\alpha\in{\mathcal F}(\tilde\Lambda)$ and $x_\lambda\in D\big(\lambda;{\bf s}(\lambda)\big)$.
In particular, $DS(\tilde\Lambda)-n\subset DS(\tilde\lambda\setminus\alpha)$ and hence it is indeed in $p$, proving (\ref{eq_lemma_dsystemsinidempotents}).
This concludes the proof of the claim that $K$ is a non-empty compact semigroup.
Thus by Ellis's lemma, it contains an idempotent. This finishes the proof of the first direction.

Now we prove the converse: assume that $p\in U$ is idempotent and let $A\in p$.
  Let $B=\{n\in A:A-n\in p\}$ and observe that $B\in p$ and for every $n\in p$ also $B-n\in p$.

  Next enumerate $\Lambda_t=\{\lambda_1,\lambda_2,\dots\}$ and let $\lambda_i=(m_i,\F^{(i)},c_i)$.
  It follows from Theorem \ref{thm_finitistic} that there exists ${\bf s}_1\in G^{m_1+1}$ such that $D(m_1,\F^{(1)},c_1;{\bf s}_1)\subset B$.

  One can construct a sequence ${\bf s}_1,{\bf s}_2,\dots$ inductively such that for all $n\in\N$, one has ${\bf s}_n\in G^{m_n+1}$ and
  \begin{equation}\label{eq_thm_dsystemsgeneral}\forall \alpha\subset[n], \alpha\neq\emptyset\qquad\qquad\sum_{i\in\alpha}D(m_i,\F^{(i)},c_i;{\bf s}_i)\subset B\end{equation}
  The procedure for constructing this sequence is the same as in the proof of Theorem \ref{thm_DLambdamonochromatic} and will therefore be omitted.
  The sequence ${\bf s}_1,{\bf s}_2,\dots$ satisfies (\ref{eq_thm_dsystemsgeneral}) so $B$ is a $\Lambda_t$-system, concluding the proof.
  \end{proof}

We can now give the proof of Theorem \ref{thm_systemspartitionregularity}:

\begin{proof}[Proof of Theorem \ref{thm_systemspartitionregularity}]
  Let $A\subset G$ be a $\Lambda_t$-system.
  By Lemma \ref{lemma_dsystemsinidempotents}, there exists an idempotent ultrafilter $p\in\beta G$ such that $A\in p$ and every element of $p$ is $\Lambda_t$-rich.
  For any finite partition of $A$, one of the pieces must still be in $p$.
  Invoking again Lemma \ref{lemma_dsystemsinidempotents} we deduce that every element of $p$ is a $\Lambda_t$-system, finishing the proof.
\end{proof}

\section{Applications to systems of equations}\label{section_rado}
In this section we derive some corollaries of our results that pertain to partition regularity of homogeneous systems of equations.
In particular we show that the sufficient condition in Rado's theorem, when appropriately formulated, applies to any countable commutative semigroup.
Our departure point is Rado's theorem itself.
\begin{definition}\label{def_columnscond}
  Let $d,k\in\N$, let $A$ be a $k\times d$ matrix with integer coefficients and let $c_1,\dots,c_d\in\Z^k$ be the columns of $A$. We say that $A$ satisfies the \emph{columns condition} if there exist $m\in\N$ and integers $0=d_0<d_1<d_2<\cdots<d_m<d_{m+1}=d$ such that for every $0\leq j\leq m$, the sum
  $$c_{d_j+1}+c_{d_j+2}+\cdots+c_{d_{j+1}}$$
  is in the linear span (over $\Q$) of the set $\{c_i:i\leq d_j\}$ (with the understanding that the only vector in the linear span of the empty set is $\bf 0$).
\end{definition}
\begin{theorem}[Rado \cite{Rado33}]\label{thm_rado}
  Let $d,k\in\N$ and let $A$ be a $k\times d$ matrix with integer entries.
  Then for any finite coloring of $\N$ there exists ${\bf x}=(x_1,\dots,x_d)\in\N^d$ with all coordinates in the same color and $A{\bf x}={\bf0}$ if and only if $A$ satisfies the columns condition (possibly after some permutation of the columns of $A$).
\end{theorem}

The `if' direction of Rado's theorem follows directly from Deuber's Theorem \ref{thm_introdeuber}.
The idea is that the columns condition implies the existence of a triple $(m,p,c)\in\N^3$ such that any $(m,p,c)$-set contains a solution to $A{\bf x}={\bf0}$.

More precisely, using the columns condition one can find a $d\times(m+1)$ matrix $B$ such that $AB=0$ and, for any ${\bf s}\in\N^{m+1}$, the entries of the vector $B{\bf s}$ are contained in the $(m,p,c)$-set $D(m,p,c;{\bf s})$ for some $c,p\in\N$ that only depend on $A$.
Then, for any finite coloring of $\N$ one can find ${\bf s}\in\N^{m+1}$ such that $D(m,p,c;{\bf s})$ is monochromatic, and in particular, all coordinates of $B{\bf s}$ are monochromatic.
Since $AB=0$, also $A(B{\bf s})=0$.
The details of this deduction can be found, for instance, in \cite{Graham_Rothschild_Spencer80}.

We now turn to linear systems of equations in countable commutative semigroups and establish an analogue of the columns condition in this setting.
\begin{definition}\label{def_gencolumnscond}
  Let $G$ be a countable commutative semigroup with identity $0$, let $k,d\in\N$ and let $A:G^d\to G^k$ be a homomorphism.
  For each $i=1,\dots,d$ let $c_i:G\to G^k$ be the map defined by $c_i(x)=A(0,\dots,0,x,0,\dots,0)$,where the $x$ appears in the $i$-th position.

  We say that $A$ satisfies the \emph{columns condition} if there exist $c\in\E(G)$, $m\in\N$ and $0=d_0<d_1<\cdots<d_{m+1}=d$ such that
  \begin{enumerate}
  \item The composition $(c_1+c_2+\cdots+c_{d_1})\circ c$ is the zero map;
  \item For each $1\leq t\leq m$ there are $f_1^{(t)},\dots,f_{d_i}^{(t)}\in\E(G)$ such that
  \begin{equation}\label{eq_linearcomb}(c_{d_t+1}+\cdots+c_{d_{t+1}})\circ c+\Big(c_1\circ f_1^{(t)}+\cdots+c_{d_t}\circ f_{d_t}^{(t)}\Big)=0\end{equation}
  \end{enumerate}
  \end{definition}
  This definition can be seen as a direct extension of Definition \ref{def_columnscond}.
  Indeed, when $G=\Z$, the only homomorphisms are multiplication by a fixed integer and equation (\ref{eq_linearcomb}) expresses the fact the sum $c_{d_t+1}+\cdots+c_{d_{t+1}}$ is a linear combination of $c_1,\dots,c_{d_t}$.

The next proposition is an extension of the `if' part of Rado's theorem to countable commutative semigroups.
\begin{proposition}\label{prop_columns}
Let $G$ be a countable commutative semigroup with identity $0$, let $k,d\in\N$ and let $A:G^d\to G^k$ be a homomorphism which satisfies the columns condition for some $c\in\E(G)$ that is either in the center of $\E(G)$ or is IP-regular.
Then for any finite coloring of $G$ there exists ${\bf x}=(x_1,\dots,x_d)$ with all entries in the same color such that $A({\bf x})={\bf0}$.
\end{proposition}

\begin{proof}
Let $m\in\N$ and $c\in\E(G)$ be given by the columns condition.
For each $j=1,\dots,m$ let
$$F_j=\left\{f:\big(s_0,\dots,s_{j-1}\big)\mapsto\sum_{\ell=0}^{j-1}f_i^{(m-\ell)}(s_\ell):d_{m-j}<i\leq d_{m+1-j}\right\}$$
 and let $\F=(F_1,\dots,F_m)$.
Assume we are given a finite coloring of $G$.
Appealing to either Theorem \ref{thm_finitistic} or Theorem \ref{thm_generalmain} (according to whether $c$ is in the center of $\E(G)$ or IP-regular) we can find ${\bf s}\in G^{m+1}$ such that the $(m,\F,c)$-set $D(m,\F,c;{\bf s})$ is monochromatic.
For each $i=1,\dots,d$, let $j\in\{0,\dots,m\}$ be such that $d_{m-j}<i\leq d_{m+1-j}$ (and observe that $j$ is uniquely determined).
Let
$$x_i=\sum_{\ell=0}^{j-1}f_i^{(m-\ell)}(s_\ell)+c(s_j)$$
Observe that $x_i\in D(m,\F,c;{\bf s})$ and hence all the entries of the vector ${\bf x}=(x_1,\dots,x_d)\in G^d$ are of the same color.
Finally we need to check that $A({\bf x})={\bf 0}$.
Let $c_1,\dots,c_d$ be as in Definition \ref{def_gencolumnscond} and observe that each $c_i:G\to G^k$ is a homomorphism.
We have

\begin{eqnarray*}
A({\bf x})&=&\sum_{i=1}^dc_i(x_i)=\sum_{j=0}^m\sum_{i=d_{m-j}+1}^{d_{m-j+1}}c_i(x_i)\\&=&\sum_{j=0}^m\sum_{i=d_{m-j}+1}^{d_{m-j+1}}c_i\left(\sum_{\ell=0}^{j-1}f_i^{(m-\ell)}(s_\ell)+c(s_j)\right)\\&=&\sum_{j=0}^m\sum_{i=d_{m-j}+1}^{d_{m-j+1}}(c_i\circ c)(s_j)+\sum_{j=0}^m\sum_{i=d_{m-j}+1}^{d_{m-j+1}}\sum_{\ell=0}^{j-1}(c_i\circ f_i^{(m-\ell)})(s_\ell)\\&=&\sum_{\ell=0}^m\left[\sum_{i=d_{m-\ell}+1}^{d_{m-\ell+1}}(c_i\circ c)+\sum_{j=0}^{m-\ell-1}\sum_{i=d_{m-j}+1}^{d_{m-j+1}}(c_i\circ f_i^{(m-\ell)})\right](s_\ell)\\&=&\sum_{t=0}^m\left[\left(\sum_{i=d_t+1}^{d_{t+1}}c_i\right)\circ c+\sum_{i=1}^{d_t}c_i\circ f_i^{(t)}\right](s_\ell)\\&=&0\end{eqnarray*}
where the last equality follows from the columns conditions.
\end{proof}

While Proposition \ref{prop_columns} provides a quite satisfactory extension of the sufficient condition in Rado's theorem to a general setting, it is not even clear how to formulate the necessary condition.
\begin{problem}
  Let $G$ be a countable commutative semigroup, let $k,d\in\N$ and let $A:G^d\to G^k$ be a homomorphism.
  Give necessary and sufficient conditions for $A$ so that for any finite partition of $G$ there exists a non-zero ${\bf x}=(x_1,\dots,x_d)$ with all entries in the same cell of the partition and such that $A({\bf x})={\bf0}$.
\end{problem}
We conclude by remarking that an analogue of the columns condition can be concocted for polynomial equations in such a way that an analogue of Proposition \ref{prop_columns} holds, but the condition is cumbersome and so it appears to be of little practical value.
\bibliography{refs-joel}
\bibliographystyle{plain}

\end{document}